\documentclass[letterpaper,10pt,twocolumn]{IEEEtran}

\title{Cyber-Physical Systems under Attack -- Part II: Centralized and
  Distributed Monitor Design}

\title{Attack Detection and Identification in Cyber-Physical Systems
  -- Part II:\\ Centralized and Distributed Monitor Design}

\author{Fabio Pasqualetti, Florian D\"orfler, and Francesco Bullo
  \thanks{This material is based upon work supported in part by NSF grant
    CNS-1135819 and by the Institute for Collaborative Biotechnologies
    through grant W911NF-09-0001 from the U.S. Army Research Office.}
  \thanks{Fabio Pasqualetti, Florian D\"orfler, and Francesco Bullo are
    with the Center for Control, Dynamical Systems and Computation,
    University of California at Santa Barbara, {\tt
      \{fabiopas,dorfler,bullo\}@engineering.ucsb.edu}}%
}

\usepackage{amsmath}
\usepackage{amssymb}
\usepackage{mathrsfs}
\usepackage{graphicx}

\newtheorem{theorem}{Theorem}[section]
\newtheorem{lemma}[theorem]{Lemma}
\newtheorem{corollary}[theorem]{Corollary}
\newtheorem{definition}{Definition}
\newtheorem{remark}{Remark}
\newtheorem{example}{Example}

\usepackage[vlined,ruled]{algorithm2e}
\usepackage{amsmath,graphicx,epsfig,color,amsfonts,subfigure}

\newcommand\oprocendsymbol{\hbox{$\square$}}
\newcommand\oprocend{\relax\ifmmode\else\unskip\hfill\fi\oprocendsymbol}

\newenvironment{pf}{\vspace{1ex}\noindent{\itshape
    Proof:}\hspace{0.5em}} {\hfill\QED\vspace{1ex}}

\renewcommand{\theenumi}{(\roman{enumi})}

\newcommand{\setdef}[2]{\{#1 \; : \; #2\}}
\newcommand{\subscr}[2]{{#1}_{\textup{#2}}}
\newcommand{\supscr}[2]{{#1}^{\textup{#2}}}
\newcommand{\until}[1]{\{1,\dots,#1\}}

\newcommand{\blkdiag}{\textup{blkdiag}}

\usepackage{color}

\newcommand{\Ker}{\operatorname{Ker}}

\newcommand{\Image}{\operatorname{Im}}
\newcommand{\Vtar}{\mathcal{V}^*}

\newcommand{\Star}{\mathcal{S}^*}

\newcommand{\real}{\mathbb{R}}

\newcommand{\transpose}{\mathsf{T}} %or \top or \intercal

\newcommand{\Basis}{\operatorname{Basis}}

\newcommand{\mc}{\mathcal}

\renewcommand{\baselinestretch}{0.974}
\begin{document}
\maketitle

\begin{abstract}
  Cyber-physical systems integrate computation, communication, and
  physical capabilities to interact with the physical world and
  humans. Besides failures of components, cyber-physical systems are
  prone to malicious attacks so that specific analysis tools and
  monitoring mechanisms need to be developed to enforce system
  security and reliability. This paper builds upon the results
  presented in our companion paper \cite{FP-FD-FB:12a} and proposes
  centralized and distributed monitors for attack detection and
  identification. First, we design optimal centralized attack
  detection and identification monitors. Optimality refers to the
  ability of detecting (respectively identifying) \emph{every}
  detectable (respectively identifiable) attack. Second, we design an
  optimal distributed attack detection filter based upon a waveform
  relaxation technique. Third, we show that the attack identification
  problem is computationally hard, and we design a sub-optimal
  distributed attack identification procedure with performance
  guarantees. Finally, we illustrate the robustness of our monitors to
  system noise and unmodeled dynamics through a simulation study.
\end{abstract}

\section{Introduction}
\emph{Cyber-physical systems} 
% are among the most critical infrastructures, and they
need to remain functional and operate reliably in presence of
unforeseen failures and, possibly, external attacks.
% Cyber-physical systems arise from the tight integration of physical
% processes, computational resources, and communication capabilities.
Besides failures and attacks on the physical infrastructure,
cyber-physical systems are also prone to cyber attacks against their
data management, control, and communication layer
% , as demonstrated by recent studies and real-world incidents
\cite{ARM-RLE:10,KH:11,JS-MM:07,GEA-DML:05}.
% Because of the crucial role of cyber-physical systems in everyday
% life, genuine faults, malicious attacks, and security breaches need
% to be promptly detected, identified, and isolated.

In several cyber-physical systems, including water and gas
distribution networks, electric power systems, and dynamic Leontief
econometric models, the physical dynamics include both differential
equations as well as algebraic constraints.  In \cite{FP-FD-FB:12a} we
model cyber-physical systems under attack by means of linear
continuous-time differential-algebraic systems; we analyze the
fundamental limitations of attack detection and identification, and we
characterize the vulnerabilities of these systems by graph-theoretic
methods. In this paper we design monitors for attack detection and
identification for the cyber-physical model presented in
\cite{FP-FD-FB:12a}.

\noindent
\textbf{Related work.}  Concerns about security of control,
communication, and computation systems are not recent as testified by
the numerous works in the fields of fault-tolerance control and
information security. However, as discussed in \cite{FP-FD-FB:12a},
cyber-physical systems feature vulnerabilities beyond fault-tolerance
control and information security methods.

Attack detection and identification monitors have recently been
proposed.
% The design of attack detection and identification monitors is a key
% step to enhance security and reliability of cyber-physical
% systems. Because this problem has received particular attention only
% recently, few design procedures for attack detection and
% identification have been proposed.
In \cite{SS-CH:10a,FP-AB-FB:09b} monitoring procedures are designed
for the specific case of state attacks against discrete-time
nonsingular systems. In \cite{FH-PT-SD:11} an algorithm to detect
output attacks against discrete-time nonsingular systems is described
and characterized. In \cite{YM-BS:10a} a detection scheme for replay
attacks is proposed.  Fault detection and identification schemes for
linear differential-algebraic power network models are presented in
\cite{ES:04,ADDG-ST:10} and in the conference version of this paper
\cite{FP-FD-FB:11i}. We remark that the designs in
\cite{ES:04,ADDG-ST:10} consider particular known faults rather than
unknown and carefully orchestrated cyber-physical attacks.  Finally,
protection schemes for output attacks against systems described by
purely static models are presented, among others, in
\cite{DG-HS:10,RBB-KMR-QW-HK-KN-TJO:10}.

\noindent
\textbf{Contributions.} The main contributions of this work are as
follows. 
First, for the differential-algebraic model of cyber-physical systems
under attacks developed in \cite{FP-FD-FB:12a}, we design centralized
monitors for attack detection and identification. With respect to the
existing solutions, in this paper we propose attack detection and
identification filters that are effective against both state and
output attacks against linear continuous-time differential-algebraic
cyber-physical systems. Our monitors are designed by using tools from
geometric control theory; they extend the construction of
\cite{MAM-GCV-ASW-CM:89} to descriptor systems with direct feedthrough
matrix, and they are guaranteed to achieve optimal performance, in the
sense that they detect (respectively identify) every detectable
(respectively identifiable) attack.

Second, we develop a fully distributed attack detection filter with
optimal (centralized) performance. Specifically, we provide a
distributed implementation of our centralized attack detection filter
based upon iterative local computations by using the Gauss-Jacobi
waveform relaxation technique. For the implementation of this method,
we rely upon cooperation among geographically deployed control
centers, each one responsible for a part of the system. In particular,
we require each control center to have access to the measurements of
its local subsystem, synchronous communication among neighboring
control centers at discrete time instants, and ability to perform
numerical integration.

% Compared to our earlier work \cite{FP-FD-FB:11i}, the proposed
% centralized detection filter is sparse and thus amenable to a
% distributed implementation. By means of a decentralized representation
% of the system dynamics, we provide a distributed implementation of the
% attack detection filter based upon iterative local computations using
% the Gauss-Jacobi waveform relaxation technique. In the end, we propose
% a fully distributed attack detection filter whose implementation

% requires communication among the local control centers only at
% discrete time instants and which achieves the guaranteed centralized
% performance.

Third, we show that the attack identification problem is inherently
computationally hard. Consequently, we design a distributed
identification method that achieves identification, at a low
computational cost and for a class of attacks, which can be
characterized accurately. Our distributed identification methods is
based upon a \emph{divide and conquer} procedure, in which first
corrupted regions and then corrupted components are identified by
means of local identification procedures and cooperation among
neighboring regions. Due to cooperation, our distributed procedure
provably improves upon the fully decoupled approach advocated in
decentralized control \cite{MS-YG:92}.

Fourth, we present several illustrative examples. Besides
illustrating our findings concerning centralized and distributed
detection and identification, our numerical investigations show that our
methods are effective also in the presence of system noise,
nonlinearities, and modeling uncertainties.

Finally, as a minor contribution, we build upon the estimation method
in \cite{FJB-TF-WP-GZ:11} to characterize the largest subspace of the
state space of a descriptor system that can be reconstructed in the
presence of unknown inputs.

% This paper also features two minor contributions. First, we adopt the
% distributed paradigm proposed in \cite{MS-YG:92} and we develop a
% distributed identification method. This method requires only local
% knowledge of the system parameters, and no cooperation among
% neighboring monitors. Second, we build upon the estimation method in
% \cite{FJB-TF-WP-GZ:11} to characterize the largest subspace of the
% state space of a descriptor system that can be reconstructed in the
% presence of unknown inputs.

\noindent
\textbf{Paper organization.} Section \ref{sec:setup} contains a
mathematical description of the problems under investigation. 
In Section \ref{sec:detection} we design monitors
for attack detection. Specifically, we propose optimal centralized,
decentralized, and distributed monitors. In Section
\ref{sec:identification} we show that the attack identification
problem is computationally hard. Additionally, we design an optimal
centralized and a sub-optimal decentralized attack identification
monitor. Finally, Section \ref{sec:example} and Section
\ref{sec:conclusion} contain, respectively, our numerical studies, and
our conclusion.

%%% Local Variables: 
%%% mode: latex
%%% TeX-master: "main"
%%% End:

\section{Problem setup and preliminary concepts}\label{sec:setup}
In this section we recall the framework proposed in
\cite{FP-FD-FB:12a} for cyber-physical systems and attacks. We model a
cyber-physical system under attack with the time-invariant descriptor
system
\begin{align}
  \label{eq: cyber_physical_fault}
  \begin{split}
  E\dot x(t) &= Ax(t) + Bu(t),\\
  y(t)&= Cx(t) + Du(t) ,
  \end{split}
\end{align}
where $x(t) \in \real^n$, $y(t) \in \real^p$, $E \in \real^{n \times
  n}$, $A \in \real^{n \times n}$, $B \in \real^{n \times m}$, $C \in
\real^{p \times n}$, and $D \in \real^{p \times m}$. Here the matrix
$E$ is possibly singular, and the input terms $Bu(t)$ and $Du(t)$ are
unknown signals describing disturbances affecting the plant. Besides
reflecting the genuine failure of systems components, these
disturbances model the effect of an attack against the cyber-physical
system. For notational convenience and without affecting generality,
we assume that each state and output variable can be independently
compromised by an attacker. Thus, we let $B = \begin{bmatrix}I ,
  0\end{bmatrix}$ and $D = \begin{bmatrix}0 , I\end{bmatrix}$ be
partitioned into identity and zero matrices of appropriate dimensions,
and, accordingly, $u(t) = \begin{bmatrix}
  u_{x}(t)^{\transpose},u_{y}(t)^{\transpose} \end{bmatrix}^{\transpose}$.
Hence, the unknown input $(Bu(t),Du(t)) = (u_{x}(t),u_{y}(t))$ can be
classified as \emph{state attack} affecting the system dynamics and as
\emph{output attack} corrupting directly the measurements vector.

% \flomargin{I updated the notation for consistency.}
% \fpmargin{Thanks}
The attack signal $t \mapsto u(t) \in \mathbb R^{n+p}$
% $u:\, \mathbb R_{\geq 0} \to \mathbb R^{n+p}$
depends upon the specific attack strategy. In the presence of $k \in
\mathbb{N}_0$, $k \le n+p$, attackers indexed by the {\em attack set}
$K \subseteq \until{n+p}$ only and all the entries $K$ of $u(t)$ are
nonzero over time. To underline this sparsity relation, we sometimes
use $u_K(t) \in \mathbb R^{|K|}$ to denote the {\em attack mode}, that is the subvector of
$u(t)$ indexed by $K$. Accordingly, we use the pair $(B_{K},D_{K})$, where
$B_K$ and $D_K$ are the submatrices of $B$ and $D$ with columns in
$K$, to denote the {\em attack signature}. Hence, $Bu(t) = B_K u_K
(t)$, and $Du(t) = D_K u_K (t)$. 
We make the following assumptions on system \eqref{eq:
  cyber_physical_fault}, a discussion of which can be found in
\cite{FP-FD-FB:12a}:
\begin{enumerate}\setcounter{enumi}{2}
\item[(A1)] the pair $(E,A)$ is regular, that is, $\textup{det}(sE - A)$ does
  not vanish identically,
\item[(A2)] the initial condition $x(0) \in \mathbb R^{n}$ is
  consistent, that is, $(Ax(0) + B u(0)) \perp \Ker(E^{\transpose}) = 0$; and
\item[(A3)] the input signal $u(t)$ is smooth.
\end{enumerate}

The following definitions are inspired by our results in
\cite{FP-FD-FB:12a}. Let $y(x_0,u,t)$ be the output sequence generated
from the initial state $x_0$ under the attack signal $u(t)$.
% Let $T \subseteq \real_{\ge 0}$ denote the set of instants at which
% attack detection and identification is performed.

\begin{definition}{\bf\emph{(Undetectable attack
    set)}}\label{undetectable_input}
  For the linear descriptor system \eqref{eq: cyber_physical_fault},
  the attack set $K$ is \emph{undetectable} if there exist initial
  conditions $x_1 , x_2 \in \real^n$, and an attack mode $u_K (t)$
  such that, for all $t \in \real_{\ge 0}$, it holds $y(x_1,u_K,t) = y(x_2,0,t)$.
\end{definition}
\smallskip

\begin{definition}{\bf\emph{(Unidentifiable attack
      set)}}\label{unidentifiable_input}
  For the linear descriptor system \eqref{eq: cyber_physical_fault},
  the attack set $K$ is \emph{unidentifiable} if there exists an
  attack set $R$, with $|R|\le|K|$ and $R\neq{K}$, initial conditions
  $x_K , x_R \in \real^n$, and attack modes $u_K(t)$, $u_R (t)$ such
  that, for all $t \in \real_{\ge 0}$, it holds $y(x_K,u_{K},t) =
  y(x_R,u_{R},t)$.
\end{definition}
\smallskip

In our companion paper \cite{FP-FD-FB:12a} we characterize 
% present system-theoretic and graph-theoretic characterizations of
undetectable and unidentifiable attacks. In this paper, instead, we
design monitors to achieve attack detection and identification.

%%% Local Variables: 
%%% mode: latex
%%% TeX-master: "main"
%%% End:

% \section{Preliminary Notions}\label{sec:preliminaries}
% \input{preliminaries}

%\clearpage
\section{Monitor design for attack detection}\label{sec:detection}
\subsection{Centralized attack detection monitor design}
%In \cite{FP-FD-FB:11i} we presented a centralized method for the
%detection of attacks based upon Kron reduction of the algebraic
%equations in the descriptor model \eqref{eq: cyber_physical_fault}
%(see \cite[Section III.D]{FP-FD-FB:12a}). The resulting attack
%detection filter is low-dimensional and non-singular but
%non-sparse. In what follows, we present a similar but {\em sparse}
%centralized attack detection filter. This sparsity will be key to
%develop a distributed detection method.
In the following we present a centralized attack detection filter based on a modified Luenberger observer. 

\begin{theorem}{\bf\emph{(Centralized attack detection filter)}}
\label{Theorem: Centralized attack detection filter}
Consider the descriptor system \eqref{eq: cyber_physical_fault} and
assume that the attack set $K$ is detectable, and that the network
initial state $x(0)$ is known.  Consider the {\em centralized attack
  detection filter} \begin{align}
	\begin{split}
		E \dot w(t) &= (A + G C) w(t) - Gy(t) , \\
		r(t) &= C w(t) - y(t) ,
	\end{split}
	\label{eq:detection_filter}
\end{align}
where $w(0) = x(0)$ and the output injection $G \in \mathbb R^{n
  \times p}$ is such that the pair $(E,A+GC)$ is regular and
Hurwitz. Then $r (t) = 0$ at all times $t \in \real_{\ge 0}$ if and
only if $u_{K}(t) = 0$ at all times $t \in \real_{\ge 0}$. Moreover,
in the absence of attacks, the filter error $w(t) - x(t)$ is
exponentially stable.
\end{theorem}

\begin{proof}
  Consider the error $e(t) = w(t)-x(t)$ between the dynamic states of
  the filter \eqref{eq:detection_filter} and the descriptor system
  \eqref{eq: cyber_physical_fault}. The error dynamics with output
  $r(t)$ are given by
  \begin{align}
    \begin{split}
      \begin{split}
        E \dot e(t) &= (A + G C) e(t) - (B_{K} + GD_{K})u_{K}(t), \\
        r(t) &= Ce(t) - D_{K} u_{K}(t) ,
      \end{split}
    \end{split}
    \label{eq:detection_filter - error dynamics}
  \end{align}
  where $e(0) = 0$. To prove the theorem we show that the error system
  \eqref{eq:detection_filter - error dynamics} has no invariant zeros,
  that is, $r(t)=0$ for all $t \in \real_{\ge 0}$ if and only if $u_{K}(t)
  = 0$ for all $t \in \real_{\ge 0}$. Since the initial condition
  $x(0)$ and the input $u_{K}(t)$ are assumed to be consistent (A2) and
  non-impulsive (A3), the error system \eqref{eq:detection_filter -
    error dynamics} has no invariant zeros if and only if
  \cite[Proposition 3.4]{TG:93} there exists no triple $(s,\bar w,g_{K})
  \in \mathbb C \times \mathbb R^{n} \times \mathbb R^{p}$ satisfying
  \begin{equation}
    \begin{bmatrix}
      sE - (A+ G C) & B_{K} +  G  D_{K}  \\
      C & - D_{K}
    \end{bmatrix}
    \begin{bmatrix}
      \bar w \\ g_{K}
    \end{bmatrix}
    =
    \begin{bmatrix}
      0 \\ 0
    \end{bmatrix}
    \label{eq: pencil for error system - 1}
    \,.
  \end{equation}
  The second equation of \eqref{eq: pencil for error system - 1}
  yields $C \bar w =  D_{K} g_{K}$. Thus, by substituting
  $C \bar w$ by $D_{K} g_{K}$ in the first equation of
  \eqref{eq: pencil for error system - 1}, the set of equations
  \eqref{eq: pencil for error system - 1} can be equivalently
  written as
  \begin{equation}
    \begin{bmatrix}
      sE -  A & B_{K}  \\
      C & - D_{K}
    \end{bmatrix}
    \begin{bmatrix}
      \bar w \\ g_{K}
    \end{bmatrix}
    =
    \begin{bmatrix}
      0 \\ 0
    \end{bmatrix}
    \label{eq: pencil for error system - 2}
    \,.
  \end{equation}
  Finally, note that a solution $(s,-\bar w,g_{K})$ to above set of
  equations would yield an invariant zero, zero state, and zero input
  for the descriptor system \eqref{eq: cyber_physical_fault}. By the
  detectability assumption,\footnote{Due to linearity of the
    descriptor system \eqref{eq: cyber_physical_fault}, the
    detectability assumption reads as ``the attack $(B,D,u(t))$ is
    detectable if there exist no initial condition $x_0 \in
    \real^{n}$, such that $y(x_0,u,t) = 0$ for all $t \in
    \mathbb{R}_{\ge 0}$.''}  the descriptor model \eqref{eq:
    cyber_physical_fault} has no zero dynamics and the matrix pencil
  in \eqref{eq: pencil for error system - 2} necessarily has full
  rank.  It follows that the triple $(E,A,C)$ is observable, so that
  $G$ can be chosen to make the pair $(E, A + G C)$ Hurwitz
  \cite[Theorem 4.1.1]{LD:89}, and the error system
  \eqref{eq:detection_filter - error dynamics} is stable and with no
  zero dynamics. %This concludes the proof of
%  Theorem \ref{Theorem: Centralized attack detection filter}.
\end{proof}

\begin{remark}{\bf\emph{(Detection and identification filters for
      unknown initial condition and noisy
      dynamics)}}\label{remark:det_noise}
\label{Remark: detection filter for unknown initial condition}
If the network initial state is not available, then, since $(E, A+GC)$
is Hurwitz, an arbitrary initial state $w(0) \in \mathbb R^{n}$ can be
chosen. Consequently, the filter converges asymptotically, and some
attacks may remain undetected or unidentified. For instance, if the
eigenvalues of the detection filter matrix have real part smaller than
$c < 0$, with $c \in \real$, then, in the absence of attacks, the
residual $r(t)$ exponentially converges to zero with rate less than
$c$. Hence, only inputs $u(t)$ that vanish faster or equal than
$e^{-ct}$ may remain undetected by the filter
\eqref{eq:detection_filter}. %For a more involved design leading to a finite-time convergent filter see
%\cite{AVM-HTT:94,TR-FA:08}.
 Alternatively, the detection filter can be modified so as to
 converge in a predefined finite time, see
 \cite{AVM-HTT:94,TR-FA:08}. In this case, every attack signal is
 detectable after a finite transient.

If the dynamics and the measurements of \eqref{eq:
  cyber_physical_fault} are affected by modeling uncertainties and
noise with known statistics, then the output injection matrix $G$ in
\eqref{eq:detection_filter} should be chosen as to optimize the
sensitivity of the residual $r(t)$ to attacks versus the effect of
noise. Standard robust filtering or model matching techniques can be
adopted for this task \cite{SS-IP:05}. Statistical hypothesis
techniques can subsequently be used to analyze the residual $r (t)$
\cite{MB-IVN:93}. Finally, as discussed in \cite{FP-FD-FB:12a},
attacks aligned with the noise statistics turn out to be
undetectable.
%This aspect is left as the
%subject of future investigation.  
\oprocend
\end{remark}

%Observe that the design of the filter \eqref{eq:detection_filter} is
%independent of the attack $(B_{K},D_{K})$, and that, although the
%matrices $(E,A,C)$ are sparse, the filter is centralized. In what
%follows we propose a distributed implementation of the attack
%detection filter \eqref{eq:detection_filter}.

Observe that the design of the filter \eqref{eq:detection_filter} is
independent of the particular attack signature $(B_{K},D_{K})$ and its
performance is optimal in the sense that any detectable attack set $K$
can be detected.
 We remark that for index-one descriptor systems such as power system
 models, the filter \eqref{eq:detection_filter} can analogously be
 designed for the corresponding Kron-reduced model, as defined in
 \cite{FP-FD-FB:12a}. In this case, the resulting attack detection
 filter is low-dimensional and non-singular but also non-sparse, see
 \cite{FP-FD-FB:11i}.
 In comparison, the presented filter \eqref{eq:detection_filter},
 although inherently centralized, features the {\em sparse} matrices
 $(E,A,C)$.  This sparsity will be key to develop a distributed attack
 detection filter.

\subsection{Decentralized attack detection monitor
  design}\label{sec:decentralized}
  
Let $\subscr{G}{t} = (\mc V, \mathcal{E})$ be the directed graph
associated with the pair $(E,A)$, where the vertex set $V = \until n$
corresponds to the system state, and the set of directed edges
$\mathcal{E} =\{(x_j,x_i) : e_{ij}\neq 0 \text{ or } a_{ij} \neq 0\}$ is
induced by the sparsity pattern of $E$ and $A$; see also \cite[Section
IV]{FP-FD-FB:12a}. Assume that $V$ has been partitioned into $N$
disjoint subsets as $V = V_1 \cup \cdots \cup V_N$, with $|V_i| =
n_i$, and let $\subscr{G}{t}^i= (V_i , \mathcal{E}_i)$ be the $i$-th
subgraph of $\subscr{G}{r}$ with vertices $V_i$ and edges
$\mathcal{E}_i = \mathcal{E} \cap (V_i \times V_i)$.
%
%Finally, let the set of in-neighbors of $i$ be $\mc N_i^\textup{in} =
%\setdef{j \in \until{N}\setminus i}{V_{i} \times V_{j} \subset \mc
%  E}$, and the set of out-neighbors be $\mc N_i^\textup{out} =
%\setdef{j \in \until{N}\setminus i}{V_{j} \times V_{i} \subset \mc
%  E}$.
According to this partition, and possibly after relabeling the
states, the system matrix $A$ in \eqref{eq: cyber_physical_fault} can
be written\,as
\begin{align*} 
  A =
  \begin{bmatrix}
    A_{1} & \cdots & A_{1N}\\
    \vdots & \vdots & \vdots\\
    A_{N1} & \cdots & A_{N}
  \end{bmatrix}
  =
  A_{D} + A_{C},
\end{align*}
%% Florian's code:
%where $A_i \in \mathbb{R}^{n_i \times n_i}$, $A_{ij} \in
%\mathbb{R}^{n_i \times n_j}$, $A_D = \blkdiag(A_1,\dots,A_N)$, and $A_C = A -
%A_D$. Notice that $A_{D}$ is the system matrix of the isolated subsystems and $A_{C}$ describes the
%interconnection structure among the subsystems. 
%%
%% Fabio's code:
where $A_i \in \mathbb{R}^{n_i \times n_i}$, $A_{ij} \in
\mathbb{R}^{n_i \times n_j}$, $A_D$ is block-diagonal, and $A_C \!=\!
A - A_D$. Notice that, if $A_D = \blkdiag(A_1,\dots,A_N)$, then
$A_{D}$ represents the isolated subsystems and $A_{C}$ describes the
interconnection structure among the subsystems. Additionally, if the
original system is sparse, then several blocks in $A_C$ vanish.
We make the following
assumptions:
\begin{enumerate}
\item[(A4)] the matrices $E$, $C$ are block-diagonal, that is $E =
  \blkdiag(E_{1},\dots,E_{N})$, $C = \blkdiag(C_{1},\dots,C_{N})$,
  % \begin{align*}
  %   E = \blkdiag(E_{1},\dots,E_{N}),\;
  %   C = \blkdiag(C_{1},\dots,C_{N}),
  % \end{align*}
  where $E_i \in \mathbb{R}^{n_i \times n_i}$ and $C_i \in
\mathbb{R}^{p_i \times n_i}$,
\item[(A5)] each pair $(E_i,A_i)$ is regular, and each triple
  $(E_i,A_i,C_i)$ is observable.
\end{enumerate}
Given the above structure and in the absence of attacks, the
descriptor system \eqref{eq: cyber_physical_fault} can be written as
the interconnection of $N$ subsystems of the form
\begin{align}\label{eq:subsystem}
  \begin{split}
    E_i \dot x_i (t)&= A_i x_i(t) + \sum_{j \in \mc N_i^\textup{in}} A_{ij} x_j(t),\\
    y_i (t) &= C_i x_i (t), \;\; i \in \until N,
  \end{split}
\end{align}
where $x_i (t)$ and $y_i(t)$ are the state and output of the $i$-th
subsystem and $\mc N_i^\textup{in} = \setdef{j \in \until{N}\setminus
  i}{\|A_{ij}\| \neq 0}$ are the in-neighbors of subsystem $i$. We
also define the set of out-neighbors as $\mc N_i^\textup{out} =
\setdef{j \in \until{N}\setminus i}{\|A_{ji}\| \neq 0}$. We assume the
presence of a \emph{control center} in each subnetwork
$\subscr{G}{t}^i$ with the following capabilities:
\begin{enumerate}
\item[(A6)] the $i$-th control center knows the matrices $E_i$, $A_i$,
  $C_i$, as well as the neighboring matrices $A_{ij}$, $j \in \mc
  N_i^\textup{in}$; and
\item[(A7)] the $i$-th control center can transmit an estimate of its
  state to the $j$-th control center if $j \in \mc
  N_i^\textup{out}$.
\end{enumerate}
% Under assumptions (A4) -- (A7), we consider the problem of
% designing a distributed algorithm for the control centers to
% cooperatively detect cyber-physical attacks.

Before deriving a fully-distributed attack detection filter, we
explore the question of {\em decentralized stabilization} of the error
dynamics of the filter \eqref{eq:detection_filter}. For each subsystem
\eqref{eq:subsystem}, consider the local residual generator
\begin{align}
    E_i \dot w_i(t) &= (A_i + G_i C_i) w_i(t) + \sum_{j \in \mc N_i^\textup{in}} A_{ij}
    x_j(t) - G_i y_i(t), \nonumber\\
    r_i (t) &= y_i(t) - C_i w_i(t), \;\; i \in \until N,
    \label{eq:detection_filter_decentralized}
\end{align}
where $w_i(t)$ is the $i$-th estimate of $x_i(t)$ and $G_{i} \in
\mathbb R^{n_{i} \times p_{i}}$. In order to derive a compact
formulation, let $w(t) = [w_1^\transpose (t) \, \cdots \,
w_N^\transpose (t)]^\transpose$, $r(t) = [r_1^\transpose (t) \, \cdots
\, r_N^\transpose (t)]^\transpose$, and $G =
\blkdiag(G_1,\dots,G_N)$. Then, the overall filter dynamics
\eqref{eq:detection_filter_decentralized} are
\begin{align}\label{eq:detection_filter_decentralized_vector_notation}
\begin{split}
  E \dot w(t) &= (A_D + GC) w(t) + A_C w(t) - G y(t) \,,
  \\ r(t) &= y(t) - C w(t) \,.
  \end{split}
\end{align}
Due to the observability assumption (A5) an output injection matrix
$G_i$ can be chosen such that each pair $(E_{i},A_{i}-G_{i}C_{i})$ is Hurwitz
\cite[Theorem 4.1.1]{LD:89}. Notice that, if each pair $(E_i,A_i +
G_iC_i)$ is regular and Hurwitz, then $(E, A_D + GC)$ is also regular
and Hurwitz since the matrices $E$ and $A_D + GC$ are
block-diagonal. We are now ready to state a condition for the
decentralized stabilization of the filter
\eqref{eq:detection_filter_decentralized_vector_notation}.

% \flomargin{this theorem considers the system not(!) under
% attack. either we state this explicitly or we go through the general
% case}
%\begin{lemma}\emph{\bf (Decentralized
%    stabilization):}\label{Lemma:decentralized_estimation}
%  Consider the filter dynamics
%  \eqref{eq:detection_filter_decentralized_vector_notation}, and let
%  $G = \blkdiag(G_1,\dots,G_N)$ be such that $(E,A_D + GC)$ is regular
%  and Hurwitz. The filter error $x(t) - w(t)$ is asymptotically stable if
%  \begin{equation}
%     \rho \left( (j\omega E-A_D-GC)^{-1} A_C  \right) < 1
%    \text{ for all }\omega \in \real \,,
%    \label{eq:decentralized_stabilization_condition}
%  \end{equation}
%  where $\rho(\cdot)$ denotes the spectral radius operator.
%\end{lemma}
%
% \flomargin{rewrote theorem , the prior version considers the system
% not(!) under attack}
\begin{lemma}{\bf\emph{(Decentralized stabilization of the attack
    detection filter)}}\label{Lemma:decentralized_estimation}
Consider the descriptor system \eqref{eq: cyber_physical_fault}, and
assume that the attack set $K$ is detectable and that the network
initial state $x(0)$ is known.
  Consider the attack detection filter
  \eqref{eq:detection_filter_decentralized_vector_notation}, where
  $w(0) = x(0)$ and $G = \blkdiag(G_1,\dots,G_N)$ is such that $(E,A_D
  + GC)$ is regular and Hurwitz. Assume that
  \begin{equation}
     \rho \left( (j\omega E-A_D-GC)^{-1} A_C  \right) < 1
    \text{ for all }\omega \in \real \,,
    \label{eq:decentralized_stabilization_condition}
  \end{equation}
  where $\rho(\cdot)$ denotes the spectral radius operator.  Then $r
  (t) = 0$ at all times $t \in \real_{\ge 0}$ if and only if $u_{K}(t)
  = 0$ at all times $t \in \real_{\ge 0}$. Moreover, in the absence of
  attacks, the filter error $w(t) - x(t)$ is exponentially stable.
\end{lemma}

\begin{proof}
The error $e(t) = w(t) - x(t)$ obeys the dynamics
\begin{align}
    E \dot e(t) &= (A_D + A_C  + GC) e(t) - (B_{K}+GD_{K}) u_{K}(t),
    \nonumber \\
    r(t) &= C e(t) - D_{K} u_{K}(t)
    \label{eq:error_system} \,.
\end{align}
A reasoning analogous to that in the proof of Theorem \ref{Theorem:
  Centralized attack detection filter} shows the absence of zero
dynamics. Hence, for $r (t) = 0$\,at\,all times $t \in \real_{\ge 0}$ if
and only if $u_{K}(t) = 0$ at all times $t \in \real_{\ge 0}$.

% If the initial condition is not known, we need to guarantee
% asymptotic stability of the error dynamics $e(t)$.
To show stability of the error dynamics in the absence of attacks, we employ the small-gain
approach to large-scale interconnected systems \cite{MV:81} and
rewrite the error dynamics \eqref{eq:error_system} as the closed-loop
interconnection of the two subsystems
  \begin{align*}
    \Gamma_{1}:& \quad E \dot e(t) = (A_{D} + GC) e(t) + v(t) \,, \\
     \Gamma_{2}:& \quad v(t) = A_{C}e(t)   \,.
  \end{align*}
  % The generalized eigenvalues of the pair $(E, A_D + GC)$ are
  % Hurwitz since the matrices $E$ and $A_D + GC$ are block-diagonal
  % and each pair $(E_i,A_i + G_iC_i)$ is Hurwitz.
  Since both subsystems $\Gamma_{1}$ and $ \Gamma_{2}$ are causal and
  internally Hurwitz stable, the overall error dynamics
  \eqref{eq:error_system} are stable if the loop transfer function
  $\Gamma_{1}(j \omega) \cdot \Gamma_{2}$ satisfies the spectral
  radius condition
  % small-gain condition $\|\Gamma_{1}\|_{p} \cdot
  % \|\Gamma_{2}\|_{p}\| < 1$ for some induced $p$-norm \dots leads to
  % $ \| (sE - A_D - GC)^{-1} \|_{p} \| A_C \|_{p} < 1$
  $\rho(\Gamma_{1}(j \omega) \cdot \Gamma_{2}) < 1$ for all $\omega
  \in \real$ \cite[Theorem 4.11]{SS-IP:05}. The latter condition is
  equivalent to \eqref{eq:decentralized_stabilization_condition}.
\end{proof}

Observe that, although control centers can compute the output
injection matrix independently of each other, an implementation of the
decentralized attack detection filter
\eqref{eq:detection_filter_decentralized_vector_notation} requires
control centers to continuously exchange their local estimation
vectors. Thus, this scheme has high communication cost, and it may not
be broadly applicable. A solution to this problem is presented in the
next section.

\subsection{Distributed attack detection monitor design}
In this subsection we exploit the classical waveform relaxation method
to develop a fully distributed variation of the decentralized attack
detection filter
\eqref{eq:detection_filter_decentralized_vector_notation}.  We refer
the reader to \cite{EL-AER-ALSV:82,MLC-MDI:94}
%JW-ASV-FO-AR:85
for a comprehensive discussion of waveform relaxation methods.  The
Gauss-Jacobi waveform relaxation method applied to the system
\eqref{eq:detection_filter_decentralized_vector_notation} yields the
{\em waveform relaxation iteration}
\begin{align}
  \label{eq:waveform}
  \begin{split}
  E \dot w^{(k)} (t)&= A_D w^{(k)}(t) + A_C w^{(k-1)}(t) - G y(t)
%  \\ r(t) &= y(t) - C w^{(k)}(t)
  \end{split}
  \,,
\end{align}
where $k \in \mathbb N$ denotes the iteration index, $t \in {[0,T]}$
is the integration interval for some uniform
time horizon $T>0$, and $w^{(k)}:\,[0,T] \to \mathbb R^{n}$ is a trajectory 
with the initial condition $w^{(k)} (0) = w_0$ for each $k \in \mathbb N$. 
Notice that \eqref{eq:waveform} is a descriptor
system in the variable $w^{(k)}$ and the vector $A_C w^{(k-1)}$ is a
known input, since the value of $w(t)$ at iteration $k-1$ is used. The
iteration \eqref{eq:waveform} is said to be (uniformly) \emph{convergent} if
\begin{align*}
  \lim_{k \rightarrow \infty} \max_{t \in [0,T]} \bigl\| w^{(k)} (t) - w(t) \bigr\|_{\infty} = 0\,,
%  \lim_{k \rightarrow \infty} w^{(k)} (t) - w(t) = 0
%  \,,\quad t \in [0,T] \,,
\end{align*}
where $w(t)$ is the solution of the non-iterative dynamics
\eqref{eq:detection_filter_decentralized_vector_notation}. In order to
obtain a low-complexity distributed detection scheme, we use the
waveform relaxation iteration \eqref{eq:waveform} to iteratively
approximate the decentralized filter
\eqref{eq:detection_filter_decentralized_vector_notation}. 

We start by presenting a convergence condition for the iteration
\eqref{eq:detection_filter_decentralized_vector_notation}. Recall that
a function $f:\, \mathbb R_{\geq 0} \to \mathbb R^{p}$ is said to be
of {\em exponential order} $\beta$ if there exists $\beta \in \mathbb
R$ such that the exponentially scaled function $\tilde f:\mathbb
R_{\geq 0} \to \mathbb R^{p}$, $f(t) = f(t) e^{-\beta t}$ and all its
derivatives exist and are bounded.
An elegant analysis of the waveform relaxation iteration
\eqref{eq:waveform} can be carried out in the Laplace domain
\cite{ZZB-XY:10}, where the operator mapping $w^{(k-1)}(t)$ to
$w^{(k)}(t)$ is $(sE-A_D-GC)^{-1} A_C$.
Similar to the regular Gauss-Jacobi iteration, convergence conditions
of the waveform relaxation iteration \eqref{eq:waveform} rely on the
contractivity of the iteration operator.

\begin{lemma}{\bf\emph{(Convergence of the waveform relaxation
        \cite[Theorem 5.2]{ZZB-XY:10})}}\label{Lemma:waveform_condition} Consider
    the waveform relaxation iteration \eqref{eq:waveform}. Let the
    pair $(E,A_D + GC)$ be regular, and the initial condition $w_{0}$
    be consistent. Let $y(t)$, with $t \in {[0,T]}$, be of exponential
    order $\beta$. Let $\alpha$ be the least upper bound on the real
    part of the spectrum of $(E,A)$, and define $\sigma =
    \max\{\alpha,\beta\}$.
    The waveform relaxation method \eqref{eq:waveform} is convergent
    if
  \begin{equation}
    \rho \left( ((\sigma + j \omega) E-A_D-GC)^{-1} A_C  \right) \!< 1
    \text{ for all }\omega \in\! \real
    .
    \label{eq:waveform_convergence_condition}
  \end{equation}
\end{lemma}

In the reasonable case of bounded (integrable) measurements $y(t)$, $t
\in {[0,T]}$, and stable filter dynamics, we have that $\sigma \le 0$,
and the convergence condition
\eqref{eq:waveform_convergence_condition} for the waveform relaxation
iteration \eqref{eq:waveform} equals the condition
\eqref{eq:decentralized_stabilization_condition} for decentralized
stabilization of the filer dynamics. We now propose our distributed
attack detection filter.
\begin{theorem}{{\bf\emph{(Distributed attack detection filter)}}}
\label{Theorem: Distributed Detection}
Consider the descriptor system \eqref{eq: cyber_physical_fault} and
assume that the attack set $K$ is detectable, and that the network
initial state $x(0)$ is known. Let assumptions (A1) through (A7) be
satisfied and consider the {\em distributed attack detection filter}
 \begin{align}
   E \dot w^{(k)} (t)&= \bigl(A_D + GC \bigr) w^{(k)}(t) + A_C w^{(k-1)}(t) - G y(t) \,,\nonumber\\
   r(t) &= y(t) - C w^{(k)}(t) \,,
  \label{eq:distributed_filter}
  \end{align}
  where $k \in \mathbb N$, $t \in {[0,T]}$ for some $T>0$, $w^{(k)}(0)
  = x(0)$ for all $k \in \mathbb N$, and $G = \blkdiag(G_1,\dots,G_N)$
  is such that the pair $(E,A_{D}+GC)$ is regular, Hurwitz, and
\begin{equation}
  \rho \left( (j \omega E-A_D-GC)^{-1} A_C  \right) < 1
  \text{ for all }\omega \in \real \,.
    \label{eq:main_condition}
  \end{equation}
  Then $\lim_{k \to \infty} r^{(k)} (t) = 0$ at all times $t \in
  {[0,T]}$ if and only if $u_{K}(t) = 0$ at all times $t \in {[0,T]}$.
  Moreover, in the absence of attacks, the asymptotic filter error $\lim_{k \to
    \infty} (w^{(k)}(t) - x(t))$ is exponentially stable for $t \in
  {[0,T]}$.
\end{theorem}

\begin{proof}
  Since $w^{(k)}(0) = x(0)$, it follows from Lemma
  \ref{Lemma:waveform_condition} that the solution $w^{(k)}(t)$ of the
  iteration \eqref{eq:distributed_filter} converges, as $k \to
  \infty$, to the solution $w(t)$ of the non-iterative filter dynamics
  \eqref{eq:detection_filter_decentralized_vector_notation} if
  condition \eqref{eq:waveform_convergence_condition} is satisfied
  with $\sigma = 0$ (due to integrability of $y(t)$, $t \in {[0,T]}$,
  and since the pair $(E,A_{D}+GC)$ is Hurwitz). The latter condition
  is equivalent to condition \eqref{eq:main_condition}.

  Under condition \eqref{eq:main_condition} and due to the Hurwitz
  assumption, it follows from Lemma
  \ref{Lemma:decentralized_estimation} that the error $e(t) = w(t) -
  x(t)$ between the state $w(t)$ of the decentralized filter dynamics
  \eqref{eq:detection_filter_decentralized_vector_notation} and the
  state $x(t)$ of the descriptor model \eqref{eq:
    cyber_physical_fault} is asymptotically stable in the absence of
  attacks.
%  Thus, the pair $(E, A_D + A_C + GC) = (E, A +  GC)$ is Hurwitz. 
  Due to the detectability assumption and by 
   reasoning analogous to the proof of Theorem \ref{Theorem:
    Centralized attack detection filter}, it follows that the error
  dynamics $e(t)$ have no invariant zeros. This concludes the proof of
  Theorem \ref{Theorem: Distributed Detection}.
\end{proof}

\begin{remark}{\bf\emph{(Distributed attack detection)}}
\label{Remark: Distributed implementation}
  The waveform relaxation iteration \eqref{eq:waveform} can be
  implemented in the following distributed fashion. Assume that each
  control center $i$ is able to numerically integrate the descriptor system
  \begin{align}\label{eq:local_filter_iteration}
    \begin{split}
    E_i \dot w^{(k)}_i(t) 
    =& 
    (A_i + G_i C_i) w^{(k)}_i(t) 
    \\ &+  
    \sum_{j \in \mc N_i^\textup{in}} A_{ij} w^{(k-1)}_j(t) - G_i y_i(t) \,, 
%    \\ r_i^{(k)}(t) =& y_i(t) - C_i w^{(k)}_i(t)
  \end{split}
  \end{align}
  over a time interval $t \in {[0,T]}$, with initial condition
  $w_{i}^{(k)} (0) = w_{i,0}$, measurements $y_{i}(t)$, and the
  neighboring filter states $w^{(k-1)}_j(t)$ as external inputs. Let
  $w_j^{(0)} (t)$ be an initial guess of the signal $w_j(t)$. Each
  control center $i \in \until N$ performs the following operations assuming $k = 0$ at start:
  \renewcommand{\theenumi}{(\arabic{enumi})}
  \begin{enumerate}
  \item set $k := k+1$, and compute the signal $w_i^{(k)} (t)$ by integrating the local filter equation 
    \eqref{eq:local_filter_iteration},
  \item transmit $w_i^{(k)} (t)$ to the $j$-th control center if $j
    \in \mc N_i^\textup{out}$
  \item update the input $w_j^{(k)}$ with the signal received from the
    $j$-th control center, with $j \in \mc N_i^\textup{in}$, and iterate.
  \end{enumerate}
  If the waveform relaxation is convergent, then, for $k$ sufficiently
  large, the residuals $r_i^{(k)}(t) = y_i(t) - C_i w^{(k)}_i(t)$ can
  be used to detect attacks; see Theorem \ref{Theorem: Distributed
    Detection}.
  % Initially at $k=1$, each control center needs an initial guess of
  % the neighbouring filter state trajectories $w^{(0)}_j(t)$, $t \in
  % {[0,T]}$ and for $k>1$ the control centers exchange their filter
  % state trajectories $w^{(k-1)}_j(t)$ at synchronous time instants
  % $t_{k}$ (where $\{t_{k}\}_{k \in \mathbb N}$ is an increasing
  % sequence of time instants). When each filter state $w^{(k)}_i(t)$
  % converged sufficiently (for $k$ sufficiently large) then the control
  % centers exchange the residuals $r_i^{(k)}(t) = y_i(t) - C_i
  % w^{(k)}_i(t)$, which should all be zero in the absence of an attack
  % (see Theorem \ref{Theorem: Centralized attack detection filter}).
  % \fpmargin{I would write this remar as pseudo code, or at least
  %   enumerate}
  % \fpmargin{communication may not be synchronous, as long
  %   as we have an upper bound on the communication delay.}
  In summary, our distributed attack detection scheme requires
  integration capabilities at each control center, knowledge of the
  measurements $y_{i}(t)$, $t \in {[0,T]}$, as well as synchronous
  discrete-time communication between neighboring control centers.
  \oprocend
\end{remark}

\begin{remark}{\bf\emph{(Distributed filter design)}}
\label{Remark: Distributed filter design}
As discussed in Remark \ref{Remark: Distributed implementation}, the
filter \eqref{eq:distributed_filter} can be implemented in a
distributed fashion. In fact, it is also possible to design the filter
\eqref{eq:distributed_filter}, that is, the output injections $G_{i}$,
in an entirely distributed way.
%Since for any two matrices $A$ and $B$ of appropriate
%dimension it holds that $\rho(AB) \le \| AB \|_p \le \| A
%\|_\infty \| B \|_p$ for any induced $p$-norm $\|\cdot\|_{p}$, the spectral radius condition
%\eqref{eq:main_condition} can be relaxed by the following small
%gain criterion
%\begin{equation}
%  \bigl\| ((j \omega) \cdot E-A_D-GC)^{-1} \bigr\|_{p} \bigl\| A_C  \bigr\|_{p} < 1
%  \text{ for all }\omega \in \real
%  \label{eq:main_condition_small_gain}
%  \,.
%\end{equation}
Since $\rho(A) \le \| A \|_p$ for any matrix $A$ and any induced
$p$-norm, condition \eqref{eq:main_condition} can be relaxed by the
small gain criterion to
\begin{equation}
  \bigl\| (j \omega E-A_D-GC)^{-1}  A_C  \bigr\|_{p} < 1
  \text{ for all }\omega \in \real
  \label{eq:main_condition_small_gain}
  \,.
\end{equation}
With $p = \infty$,
% i.e., evaluating the row sum
% of the absolute values of the
% matrix in condition
% \eqref{eq:main_condition_small_gain},
in order to satisfy condition \eqref{eq:main_condition_small_gain}, it
is sufficient for each control center $i$ to verify the following {\em
  quasi-block diagonal dominance} condition \cite{YO-DS-TM:86}
  for each $\omega \in \real$:
\begin{equation}
  \Bigl\| (j \omega  E_{i}-A_{i}-G_{i}C_{i})^{-1} \sum\nolimits_{j=1,j \neq i}^{n} A_{ij}  \Bigr\|_{\infty} < 1
%  \text{ for all }\omega \in \real
	\label{eq:main_condition_block-diagonal-dominance}.
\end{equation}
Note that condition \eqref{eq:main_condition_block-diagonal-dominance}
can be checked with local information, and it is a
conservative relaxation of condition \eqref{eq:main_condition}.
% In summary, each control center $i$ needs to choose the output
% injection matrix $G_{i}$ such that $A_{i} + G_{i}C_{i}$ is Hurwitz
% stable and the block-diagonal dominance condition
% \eqref{eq:main_condition_block-diagonal-dominance} is satisfied.
\oprocend
\end{remark}

\subsection{Illustrative example of decentralized detection}
% \begin{figure}[tb]
%   \centering \subfigure[]{
%     \includegraphics[width=.47\columnwidth]{./img/network_118_partitioned}
%     \label{fig:118_part}
%   } \subfigure[]{
%     \includegraphics[width=.47\columnwidth]{./img/detection_wf}
%     \label{fig:detection_wf}
%   } 
%   \caption[Optional caption for list of figures]{In
%     Fig. \ref{fig:118_part}, a partition of IEEE 118 bus system into
%     $5$ areas. Each area is monitored and operated by a control
%     center. The control centers cooperate to estimate the state and to
%     assess the functionality of the whole network. In
%     Fig. \ref{fig:detection_wf} we show the residual functions
%     computed through the distributed attack detection filter
%     \eqref{eq:distributed_filter}. The attacker compromises the
%     measurements of all the generators in area 1 from time $30$ with a
%     signal uniformly distributed in the interval $[0, 0.5]$. The
%     attack is correctly detected, because the residual functions do
%     not decay to zero. For the simulation, we run $k=100$ iterations
%     of the attack detection method.}
% \end{figure}

\begin{figure}[t]
    \centering
    \includegraphics[width=1\columnwidth]{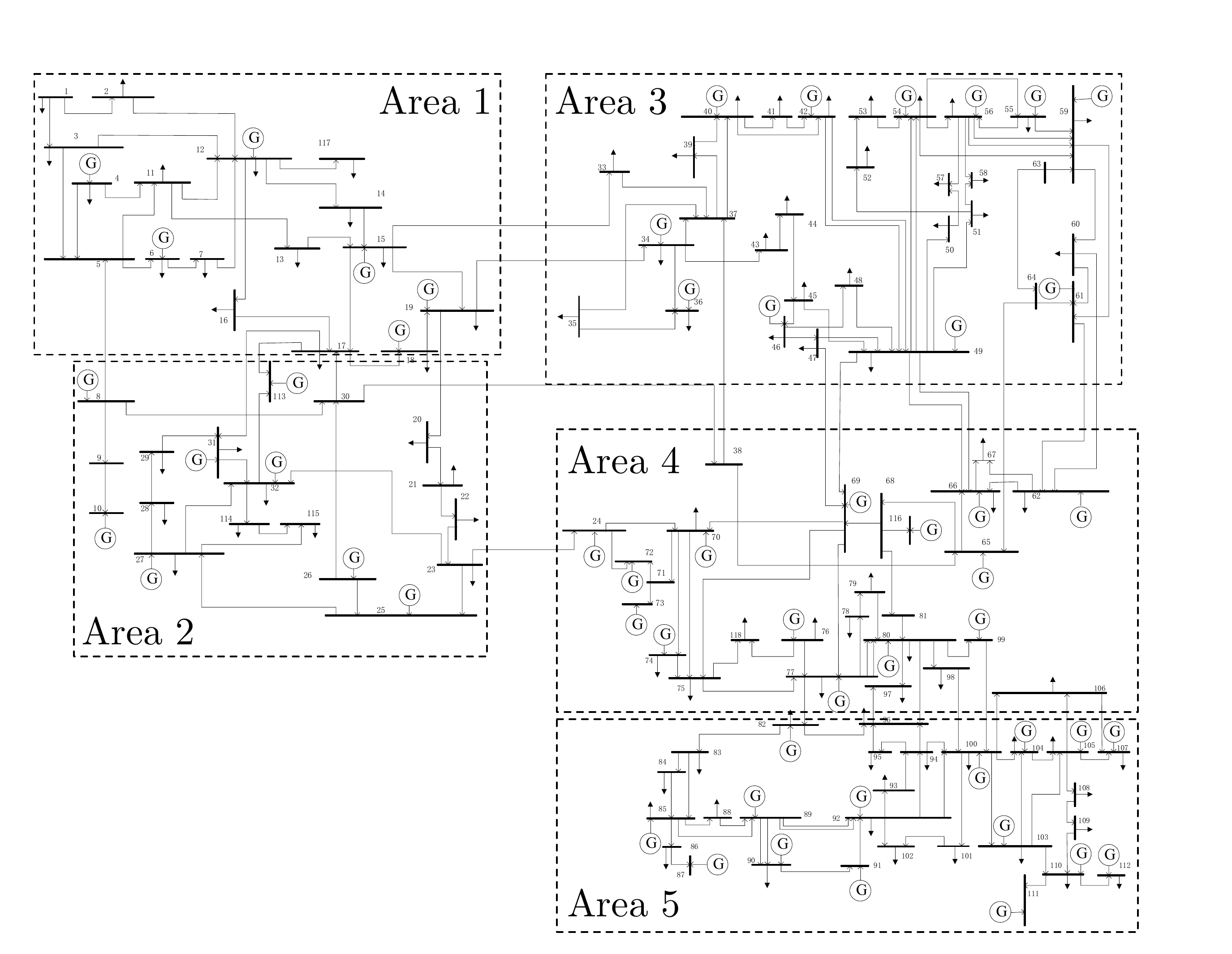}
    \caption{Partition of IEEE 118 bus system into $5$ areas. Each
      area is monitored and operated by a control center. The control
      centers cooperate to estimate the state and to assess the
      functionality of the whole network.}
    \label{fig:118_part}
\end{figure}

The IEEE 118 bus system shown in Fig. \ref{fig:118_part} represents a portion
of the Midwestern American Electric Power System as of December 1962. This test
case system is composed of 118 buses and 54 generators, and its parameters can be found, for
 example, in \cite{rdz-cem-dg:11}.
Following \cite[Section II.C]{FP-FD-FB:12a}, a linear continuous-time
descriptor model of the network dynamics under attack assumes the form
\eqref{eq: cyber_physical_fault}.
% , where the known inputs $P_x(t)$ and $P_y(t)$ are neglected in the
% forthcoming analysis due to linearity.

For estimation and attack detection purposes, we partition the IEEE
118 bus system into $5$ disjoint areas, we assign a control center to
each area, and we implement our detection procedure via the filter
\eqref{eq:distributed_filter}; see Fig. \ref{fig:118_part} for a
graphical illustration. Suppose that each control center continuously
measures the angle of the generators in its area, and suppose that an
attacker compromises the measurements of all the generators of the
first area. In particular, starting at time $30$s, the attacker adds a
signal $u_{K}(t)$ to all measurements in area 1. It can be verified
that the attack set $K$ is detectable, see
\cite{FP-FD-FB:12a}. According to assumption (A3), the attack signal
$u_{K}(t)$ needs to be continuous to guarantee a continuous state
trajectory (since the power network is a descriptor system of index 1). In
order to show the robustness of our detection filter
\eqref{eq:distributed_filter}, we let $u_{K}(t)$ be randomly
distributed in the interval $[0,0.5] \mathrm{\, rad}$.

The control centers implement the distributed attack detection
procedure described in \eqref{eq:distributed_filter}, with $G =
AC^\transpose$. It can be verified that the pair $(E,A_D + GC)$ is
Hurwitz stable, and that $\rho \left( j\omega E-A_D-GC)^{-1} A_C
\right) < 1$ for all $\omega \in \real$. As predicted by Theorem
\ref{Theorem: Distributed Detection}, our distributed attack detection
filter is convergent; see Fig. \ref{fig:detection_wf}.
\begin{figure}
  \centering
  \includegraphics[width=0.95\columnwidth]{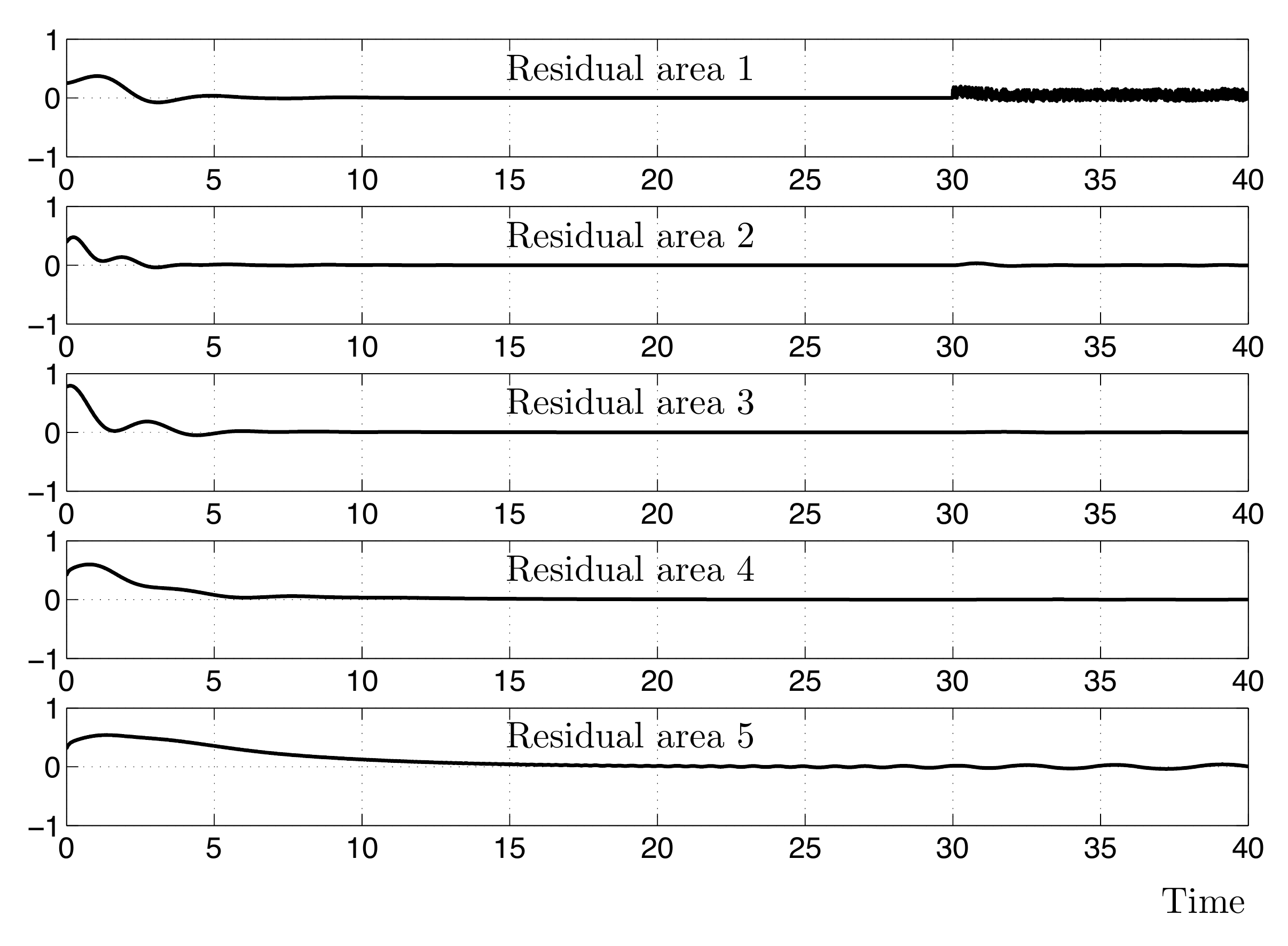}
  \caption{In this figure we show the residual functions computed
    through the distributed attack detection filter
    \eqref{eq:distributed_filter}. The attacker compromises the
    measurements of all the generators in area 1 from time $30$ with a
    signal uniformly distributed in the interval $[0, 0.5]$. The
    attack is correctly detected, because the residual functions do
    not decay to zero. For the simulation, we run $k=100$ iterations of
    the attack detection method.}
  \label{fig:detection_wf}
\end{figure}
For completeness, in Fig. \ref{fig:wf_error} we illustrate the
convergence of our waveform relaxation-based filter as a function of
the number of iterations $k$. Notice that the number of iterations
directly reflects the communication complexity of our detection
scheme.
\begin{figure}
    \centering
    \includegraphics[width=0.9\columnwidth]{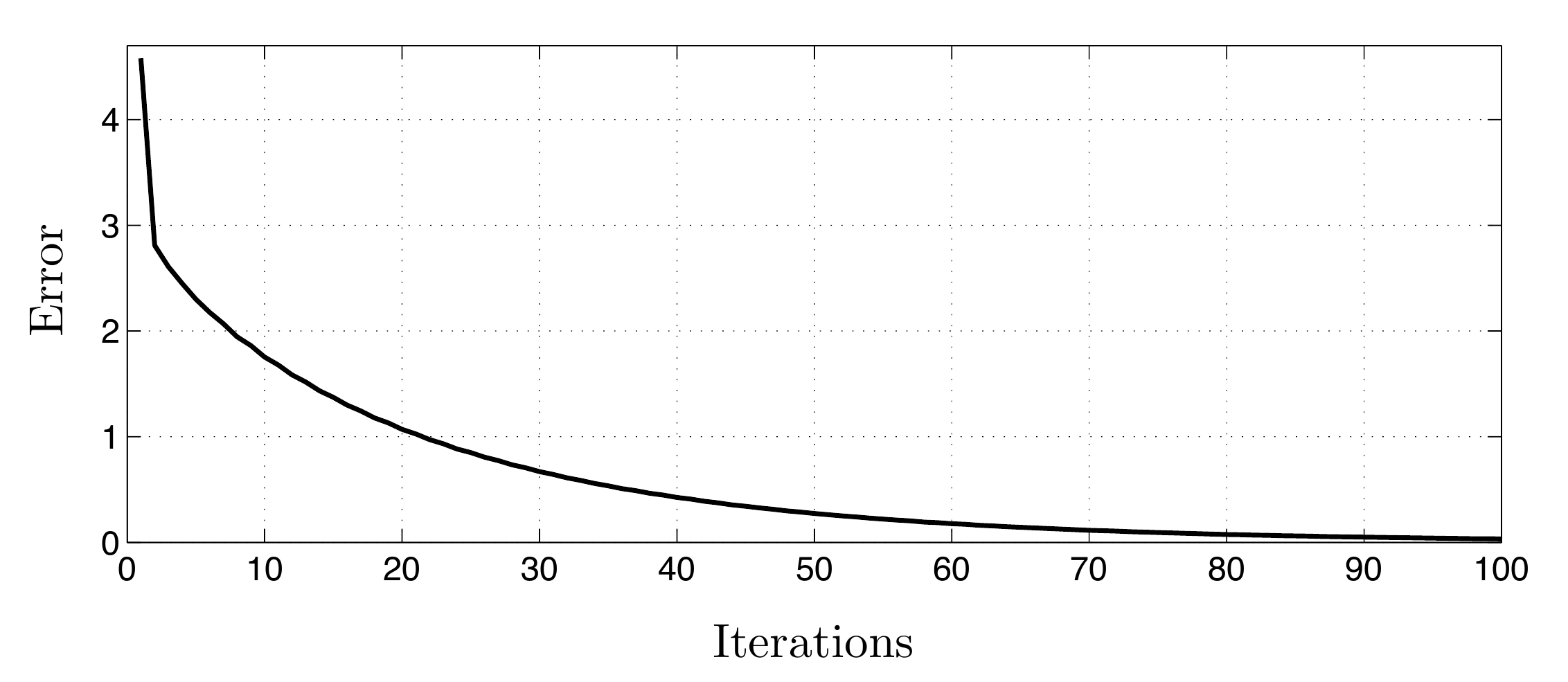}
    \caption{The plot represents the error of our waveform relaxation
      based filter \eqref{eq:distributed_filter} with respect to the
      corresponding decentralized filter
      \eqref{eq:detection_filter_decentralized_vector_notation}. Here
      the error is $\max_{t \in [0,T]} \bigl\| w^{(k)} (t) - w(t)
      \bigr\|_{\infty}$, that is, the worst-case difference of the
      outputs of the two filters. As predicted by Theorem
      \ref{Theorem: Distributed Detection}, the error is convergent.}
    \label{fig:wf_error}
\end{figure}

% In this example, since the residuals associated with the generators of
% the first area are much larger than the other residuals, the attacker
% is more likely to have corrupted the measurements of the first
% area. These important aspects of attack {\em identification} and
% regional identification will be studied in the next section.

%%% Local Variables: 
%%% mode: latex
%%% TeX-master: "main"
%%% End:

%\clearpage
\section{Monitor design for attack identification}\label{sec:identification}

\subsection{Complexity of the attack identification problem}
In this section we study the problem of attack identification, that
is, the problem of identifying from measurements the state and output
variables corrupted by the attacker.  We start our discussion by
showing that this problem is generally \emph{NP-hard}. For a vector
$x \in \mathbb R^{n}$, let $\text{supp}(x) = \{ i \in \until{n}:\, x_i
\neq 0\}$, let $\|x\|_{\ell_0} = |\text{supp}(x)|$ denote the number of
non-zero entries, and for a vector-valued signal $v:\, \mathbb{R}_{\ge
  0} \to \mathbb R^{n}$, let $\| v \|_{\mc L_{0}} = | \cup_{t \in \mathbb
  R_{\geq 0}} \text{supp}(v(t)) |$.  We consider the following
cardinality minimization problem: given a descriptor system with
dynamic matrices $E,A \in \mathbb R^{n \times n}$, measurement matrix
$C\in \mathbb R^{p \times n}$, and measurement signal $y:\,
\mathbb{R}_{\ge 0} \to \mathbb R^{p}$, find the minimum cardinality input signals
$v_{x}:\, \mathbb{R}_{\ge 0} \to \mathbb R^{n}$ and $v_{y}:\,
\mathbb{R}_{\ge 0} \to \mathbb R^{p}$ and an arbitrary initial condition
$\xi_{0}\in \mathbb{R}^n$ that explain the data $y(t)$,
that is,
%optimization problem
\begin{align}\label{prob:norm_zero}
\begin{array}{ll}
  \min\limits_{v_x, \,v_y ,\, \xi_{0}} &  \Biggl.\| v_x \|_{\mc L_{0}} + \| v_y\|_{\mc L_{0}}\\
  \textup{subject to} & E \dot \xi(t) = A \xi(t) + v_x(t),\\
  & y(t) = C\xi(t) + v_y(t),\\
  & \xi(0) = \xi_{0} \in \mathbb R^{n} \,.
\end{array}
\end{align}

\begin{lemma}{\bf\emph{(Problem
      equivalence)}}\label{lemma:equivalence_problems}
  Consider the system \eqref{eq: cyber_physical_fault} with
  identifiable attack set $K$. The optimization problem
  \eqref{prob:norm_zero} coincides with the problem of identifying the
  attack set $K$ given the system matrices $E$, $A$, $C$, and the
  measurements\,$y(t)$, where $K = \text{supp}([v_x^\transpose \;
  v_y^\transpose])$.
\end{lemma}
\begin{pf}
  Due to the identifiability of $K$, the attack identification problem
  consists of finding the smallest attack set capable of injecting an
  attack $(B_K u_K, D_K u_K)$ that generates the given measurements
  $y$ for the given dynamics $E$, $A$, $C$, and some initial
  condition; see Definition \ref{unidentifiable_input}. The statement
  follows since $B = [I,0]$ and $D = [0,I]$ in \eqref{eq:
    cyber_physical_fault}, so that $(B_K u_K, D_K u_K) = (v_x,v_y)$.
\end{pf}

% \begin{lemma}{\bf\emph{(Problem
%       equivalence)}}\label{lemma:equivalence_problems}
%   Consider the system \eqref{eq: cyber_physical_fault} with
%   identifiable attack set $K$. The optimization problem
%   \eqref{prob:norm_zero} is equivalent to the problem of identifying
%   the attack set $K$ given the system matrices $E$, $A$, $C$, and the
%   measurements\,$y(t)$.
% \end{lemma}
% \begin{pf}
%   Let $|K| = k$. Notice that there exists a state trajectory $\xi(t)$
%   and input functions $v_x (t)$, $v_x (t)$, with $\|v_x (t)\|_{\mc L_{0}} +
%   \|v_y (t)\|_{\mc L_{0}} = k$ such that $E \dot x(t) = A x(t) + v_x(t)$ and
%   $y(t) = Cx(t) + v_y(t)$. In fact, since $B = [I,0]$ and $D = [0,I]$
%   in \eqref{eq: cyber_physical_fault}, $\xi(t)$ equals the state
%   trajectory $x(t)$ of \eqref{eq: cyber_physical_fault} with input
%   $u(t) = [v_x^\transpose(t) \;
%   v_y^\transpose(t)]^\transpose$. Observe now that, since $K$ is
%   identifiable and every attack mode is nonzero, there exists no
%   attack set $R$, with $|R| < k$, that can generate the measurements
%   $y(t)$ through some attack mode, see Definition
%   \ref{unidentifiable_input}. Then, the solution to the optimization
%   problem \eqref{prob:norm_zero} is lower bounded by $k$. To conclude
%   the proof, notice that, due to the identifiability assumption, the
%   value $k$ is achieved only if the sparsity pattern of $v_x(t)$ and
%   $v_y(t)$ corresponds to the set $K$.
% \end{pf}

As it turns out, the optimization problem \eqref{prob:norm_zero}, or
equivalently our identification problem, is generally \emph{NP-hard}
\cite{MRG-DSJ:79}.
\begin{corollary}{\bf\emph{(Complexity of the attack identification
      problem)}}\label{corollary:complexity}
  Consider the system \eqref{eq: cyber_physical_fault} with
  identifiable attack set $K$. The attack identification problem given
  the system matrices $E$, $A$, $C$, and the measurements $y(t)$ is
  NP-hard.
\end{corollary}
\begin{pf}
  Consider the NP-hard \cite{EJC-TT:05} sparse recovery problem
  $\min_{\bar \xi \in \mathbb R^{n}} \| \bar y - \bar C \bar \xi
  \|_{\ell_{0}}$, where $\bar C \in \mathbb R^{p \times n}$ and $\bar
  y \in \mathbb R^{p}$ are given and constant.  In order to prove the
  claimed statement, we show that every instance of the sparse
  recovery problem can be cast as an instance of
  \eqref{prob:norm_zero}. Let $E = I$, $A = 0$, $C = \bar C$, and
  $y(t) = \bar y$ at all times. Notice that $v_y(t) = \bar y - C \xi
  (t)$ and $\xi (t) = \xi(0) + \int_0^t v_x(\tau) d\tau$. The problem
  \eqref{prob:norm_zero} can be written as
  \begin{align}\label{prob:opt2}
    \begin{array}{ll}
      &\min\limits_{v_x, \, \xi}   \;\Biggl.\| v_x \|_{\mc L_{0}} + \|
      \bar y - \bar C \xi (t) \|_{\mc L_{0}}
      \\ &=
      \min\limits_{v_x(t), \, \bar \xi}  \; \Biggl.\| v_x (t)\|_{\mc L_{0}} + \|
      \bar y - \bar C \bar \xi - \bar C\int_0^t v_x(\tau) d\tau\|_{\mc L_{0}},
    \end{array}
  \end{align}
  where $\bar \xi = \xi(0)$. Notice that there exists a minimizer to
  problem \eqref{prob:opt2} with $v_x (t) = 0$ for all $t$. Indeed,
  since $\| \bar y - \bar C \bar \xi - \bar C\int_0^t v_x(\tau)
  d\tau\|_{\mc L_{0}} = |\cup_{t \in \real_{\ge 0}} \text{supp} (\bar
  y - \bar C \bar \xi - \bar C\int_0^t v_x(\tau) d\tau)| \ge |
  \text{supp} (\bar y - \bar C \bar \xi - \bar C\int_0^0 v_x(\tau)
  d\tau)| = \| \bar y - \bar C \bar \xi \|_{\ell_{0}}$, problem
  \eqref{prob:opt2} can be equivalently written as $
  \min\nolimits_{\bar \xi} \| \bar y - \bar C \bar \xi \|_{\ell_{0}}$.
 \end{pf}

By Corollary \ref{corollary:complexity} the general attack
identification problem is combinatorial in nature, and its general
solution will require substantial computational effort. In the next
sections we propose an optimal algorithm with high computational
complexity, and a sub-optimal algorithm with low computational
complexity.
% In the next subsection, we will propose such a solution strategy as
% well as a few remedies to partially tackle these complexity issues.
We conclude this section with an example.
% To conclude this subsection, we will explore a popular heuristic to
% relax optimization problems such as \eqref{prob:norm_zero} as well as
% its short-comings when applied to a simple dynamical system.

% \flomargin{Corrected and shortened by a few lines.}
% \fpmargin{Thanks}
\begin{example}{\bf\emph{(Attack identification via $\ell_1$
      regularization)}}\label{example:regularization}
  A classical procedure to handle cardinality minimization problems of
  the form $\min_{v \in \mathbb R^{n}} \| y - Av\|_{\ell_{0}}$ is to
  use the $\ell_1$ regularization $\min_{v \in \mathbb R^{n}} \| y -
  Av\|_{\ell_{1}}$ \cite{EJC-TT:05}. 
  %DG-SB-SP:09
  This procedure can be adapted to the optimization problem
  \eqref{prob:norm_zero} after converting it into an algebraic
  optimization problem, for instance by taking subsequent
  derivatives of the output $y(t)$, or by discretizing the
  continuous-time system \eqref{eq: cyber_physical_fault} and recording several measurements.
  As shown in \cite{FH-PT-SD:11}, for discrete-time systems the
  $\ell_{1}$ regularization performs reasonably well in the presence
  of output attacks.  However, in the presence of state attacks such
  an $\ell_{1}$ relaxation performs generally poorly. In what follows,
  we develop an intuition when and why this approach fails.

% \begin{figure}[tb]
%   \centering \subfigure[]{
%     \includegraphics[width=.43\columnwidth]{./img/2rombi}
%     \label{fig:2rombi}
%   } \subfigure[]{
%     \;\;\;\includegraphics[width=.47\columnwidth]{./img/magnitude}
%     \label{fig:magnitude}
%   } 
%   \caption[Optional caption for list of figures]{In
%     Fig. \ref{fig:2rombi}, a regular consensus system $(A,B,C)$, where
%     the state variable $3$ is corrupted by the attacker, and the state
%     variables $2$, $4$, and $7$ are directly measured.  Due to the
%     sparsity pattern of $(A,B,C)$ any attack of cardinality one is
%     \emph{generically} detectable and identifiable, see
%     \cite{FP-FD-FB:12a,FP-AB-FB:09b} for further details. In
%     Fig. \ref{fig:magnitude} we plot the attack mode $\bar u (t)$ for
%     the attack set $\bar K = \{2,4,7\}$ to generate the same output as
%     the attack set $K = \{3\}$ with attack mode $u(t) = 1$. Although
%     $|\bar K| > |K|$, we have that $|\bar u_{i}(t)| < |u(t)|/3$ for $i
%     \in \{1,2,3\}$.}
% \end{figure}

  \begin{figure}
    \centering
    \includegraphics[width=.9\columnwidth]{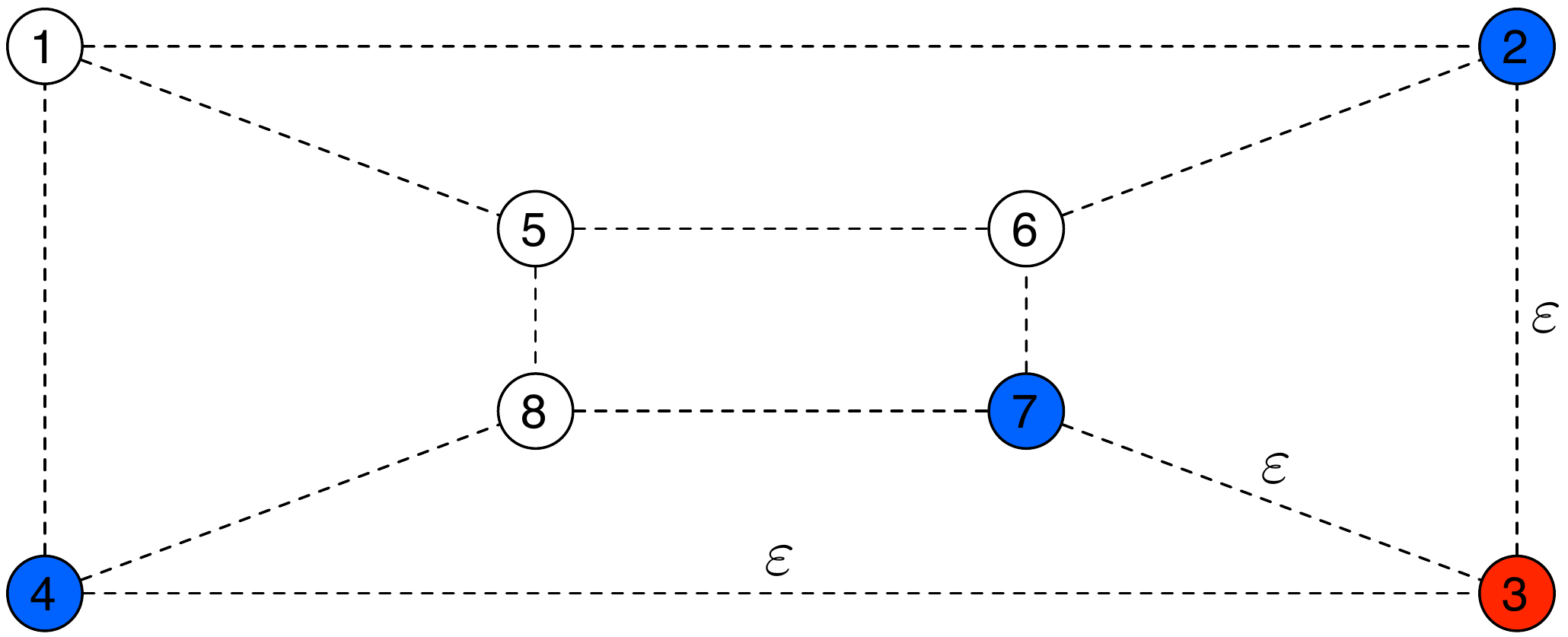}
    \caption{A regular consensus system $(A,B,C)$, where the state
      variable $3$ is corrupted by the attacker, and the state
      variables $2$, $4$, and $7$ are directly measured.  Due to the
      sparsity pattern of $(A,B,C)$ any attack of cardinality one is
      \emph{generically} detectable and identifiable, see
      \cite{FP-FD-FB:12a,FP-AB-FB:09b} for further details.}
    \label{fig:2rombi}
  \end{figure}
  Consider a consensus system with underlying network graph (sparsity
  pattern of $A$) illustrated in Fig. \ref{fig:2rombi}. The dynamics
  are described by the nonsingular matrix $E = I$ and the state matrix
  $A$
%\begin{align*}
%  \begin{split}
%    \dot x(t) &= A x(t) + B u(t),\\
%    y(t) &= C x(t),
%  \end{split}
%\end{align*}
  depending on the small parameter $0 < \varepsilon \ll
  1$\,as
\begin{equation*}
  A =
  \left[
    \begin{smallmatrix}
      -0.8  &  0.1   &      0  &  0.2 &   0.5 &        0  &       0 & 0\\
      0.1  &  -0.4-\varepsilon   & \varepsilon  &       0 &        0 &   0.3  &       0 & 0\\
      0  &  3\varepsilon   & -9\varepsilon  &  0&        0 &        0  &  6\varepsilon & 0\\
      0.1  &       0   & \varepsilon  &  -0.5-\varepsilon &        0 &        0  &       0 & 0.4\\
      0.1  &       0   &      0  &       0 &   -0.6 &   0.2  &       0 & 0.3\\
      0  &  0.4   &      0  &       0 &   0.1 &   -0.6  &  0.1 & 0\\
      0  &       0   & 3\varepsilon  &       0 &        0 &   0.4  &  -0.6-3\varepsilon & 0.2\\
      0 & 0 & 0 & 0.3 & 0.2 & 0 & 0.2 & -0.7
    \end{smallmatrix}
  \right].
\end{equation*}
The measurement matrix $C$ and the attack signature $B_{K}$ are 
\begin{equation*}
  C =
  \left[
    \begin{smallmatrix}
      0 & 1 & 0 & 0 & 0 & 0 & 0 &0\\
      0 & 0 & 0 & 1 & 0 & 0 & 0 &0\\
      0 & 0 & 0 & 0 & 0 & 0 & 1 &0
    \end{smallmatrix}
  \right]
  \;,\;
    B_{K}^{\transpose} =
  \left[
    \begin{smallmatrix}
      0 & 0 & 1 & 0 & 0 & 0 & 0 & 0
    \end{smallmatrix}
  \right]\,,
\end{equation*}
and we let $G(s) = C (sI-A)^{-1} B_{K}$.
It can be verified that the state attack $K = \{3\}$ is detectable and
identifiable.

Consider also the state attack $\bar K = \{2,4,7\}$ with signature
\begin{align*}
  B_{\bar K}^\transpose=
  \left[
    \begin{smallmatrix}
      0 & 1 & 0 & 0 & 0 & 0 & 0 &0\\
      0 & 0 & 0 & 1 & 0 & 0 & 0 &0\\
      0 & 0 & 0 & 0 & 0 & 0 & 1 &0
    \end{smallmatrix}
  \right],
\end{align*}
and let $\bar G(s) = C (sI-A)^{-1} B_{\bar K}$. We now adopt the
shorthands $u(t) \!=\! u_{K}(t)$ and $\bar u(t) \!=\!  u_{\bar K}(t)$,
and denote their Laplace transforms by $U(s)$ and $\bar U(s)$,
respectively. Notice that $\bar G(s)$ is right-invertible
\cite{GB-GM:91}. Thus, $Y (s) = G(s) U(s) = \bar G(s) \left( \bar
  G^{-1}(s) G(s) U(s) \right)$. In other words, the measurements
$Y(s)$ generated by the attack signal $U(s)$ can equivalently be
generated by the signal $\bar U(s) = \bar G^{-1}(s) G(s) U(s)$.
Obviously, we have that $\|\bar u\|_{\mc L_{0}} = 3 > \| u \|_{\mc L_{0}} = 1$, that
is, the attack set $K$ achieves a lower cost than 
$\bar K$ in the optimization problem\,\eqref{prob:norm_zero}.

Consider now the numerical realization $\varepsilon = 0.0001$, $x(0) =
0$, and $u(t) = 1$ for all $t \in \mathbb R_{\ge 0}$. The
corresponding attack mode $\bar u(t)$ is shown in
Fig. \ref{fig:magnitude}. Since $|\bar u_{i}(t)| < 1/3$ for $i \in
\{1,2,3\}$ and $t \in \mathbb R_{\geq 0}$, it follows that
%$\|u(t)\|_{\ell_p} = |u(t)| \!>\! \|\bar u(t)\|_{\ell_1} \geq \|\bar u(t)\|_{\ell_p}$ 
$\|u(t)\|_{\ell_p}  \!>\! \|\bar u(t)\|_{\ell_p}$ 
point-wise in time and  
%$\|\bar u(t)\|_{\mc L_{q}/\ell_{p}} \!<\!  \|\bar u(t)\|_{\mc L_{p}/\ell_{1}} \!<\! \| u(t) \|_{\mc L_{q}/\ell_{p}}$, 
$\|u(t)\|_{\mc L_{q}/\ell_{p}} \!>\!  \|\bar u(t)\|_{\mc
  L_{p}/\ell_{q}}$, where $p,q \geq 1$ and $\|u(t)\|_{\mc
  L_{q}/\ell_{p}} = \bigl( \int_{0}^{\infty} (\sum_{i=1}^{n+p}
|u_{i}(\tau)|^{p})^{q/p} d\tau \bigr)^{1/q}$ is the $\mc L_{q} /
\ell_{p}$-norm. Hence, the attack set $\bar K$ achieves a lower cost
than $K$ for any algebraic version of the optimization
problem \eqref{prob:norm_zero} penalizing a $\ell_{p}$ cost point-wise
in time or a $\mc L_{q} / \ell_{p}$ cost over a time
interval. Since $\|\bar u\|_{\mc L_{0}} \!>\! \| u\|_{\mc L_{0}} $, we
conclude that, in general, the identification problem cannot be solved
by a point-wise $\ell_{p}$ or $\mc L_{q} / \ell_{p}$ regularization
for any $p,q \geq 1$.
\begin{figure}
    \centering
    \includegraphics[width=1\columnwidth]{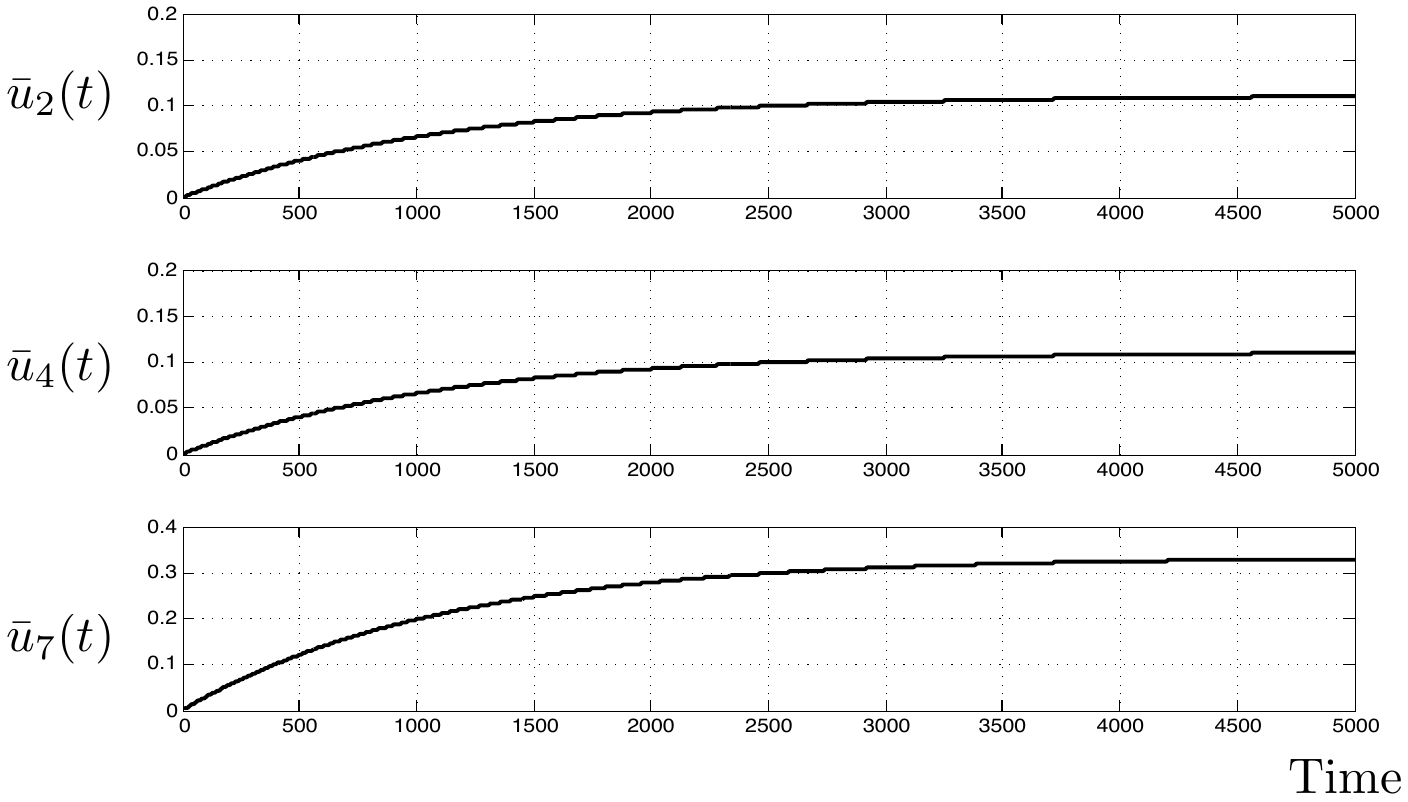}
    \caption{Plot of the attack mode $\bar u (t)$ for the attack set
      $\bar K = \{2,4,7\}$ to generate the same output as the attack
      set $K = \{3\}$ with attack mode $u(t) = 1$. Although $|\bar K|
      > |K|$, we have that $|\bar u_{i}(t)| < |u(t)|/3$ for $i \in
      \{1,2,3\}$.}
    \label{fig:magnitude}
\end{figure}

Notice that, for any choice of network parameters, a value of
$\varepsilon$ can be found such that a point-wise $\ell_{p}$ or a $\mc
L_{q} / \ell_{p}$ regularization procedure fails at identifying the
attack set. Moreover, large-scale stable systems often exhibit this
behavior independently of the system parameters. This can be easily
seen in discrete-time systems, where a state attack with attack set
$K$ affects the output via the matrix $CA^{r-1}B_{K}$, where $r$ is
the relative degree of $(A,B_{K},C)$. Hence, if $A$ is Schur stable
and thus $\lim_{k \to \infty}A^{k}=0$, then $CA^{r-1}B_{K}$ converges
to the zero matrix for increasing relative degree. In this case, an
attack closer to the sensors may achieve a lower $\mc L_{q} /
\ell_{p}$ cost than an attack far from sensors independently of the
cardinality of the attack set. In short, the $\epsilon$-connections in
Fig. \ref{fig:2rombi} can be thought of as the effect of a large
relative degree in a stable system.  \oprocend
\end{example}

% Motivated by this discussion and by Example
% \ref{example:regularization}, we present identification procedures
% that either achieve exact attack identification at high computational
% cost, or achieve partial guaranteed identification at low
% computational cost.

\subsection{Centralized attack identification monitor
  design}\label{sec:identification_centralized}
As previously shown, unlike the detection case, the identification of
the attack set $K$ requires a combinatorial procedure, since, a
% \flomargin{isn't it $\binom{n+p}{|K|}$, that is, we neglected the
% possible output attacks here?}  \fpmargin{You are right}
priori, $K$ is one of the $\binom{n+p}{|K|}$ possible attack sets. The
following centralized attack identification procedure consists of
designing a residual filter to determine whether a predefined set
coincides with the attack set.
The design of this residual filter consists of three steps -- an input
output transformation (see Lemma \ref{Lemma: removal_of_D}), a state
transformation to a suitable conditioned-invariant subspace (see Lemma
\ref{lemma:Input decoupled system representation}), and an output
injection and definition of a proper residual (see Theorem
\ref{Theorem: attack id filter}).

As a first design step, we show that the identification problem can be carried out for a
modified system without corrupted measurements, that is, without the feedthrough matrix $D$.
\begin{lemma}{\bf\emph{(Attack identification with safe
      measurements)}}
\label{Lemma: removal_of_D}
Consider the descriptor system \eqref{eq: cyber_physical_fault} with
attack set $K$. The attack set $K$ is identifiable for the descriptor
system \eqref{eq: cyber_physical_fault} if and only if it is
identifiable for the following descriptor\,system:
\begin{align}
	E \dot x(t) &= (A - B_K D_K^\dag C) x(t) + B_K (I-D_K^\dag D_K) u_K (t),\nonumber\\
    \tilde y (t) &= (I-D_{K}D_{K}^{\dag}) C x(t).
\label{eq: cyber_physical_fault_no_D}
\end{align}
\end{lemma}

\begin{proof}
  Due to the identifiability hypothesis, there
  exists no attack set $R$ with $|R| \leq |K|$ and $R \neq K$, $s \in
  \mathbb{C}$, $g_K \in \mathbb{R}^{|K|}$, $g_R \in \mathbb{R}^{|R|}$,
  and $x \in \mathbb{R}^{n} \setminus \{0\}$ such that 
  \begin{equation}
    \left[\begin{array}{c|c|c}
        sE - A & - B_{K} & -B_{R} 
        \\\hline
        C & D_{K} & D_{R}
        \\\hline
        C & D_{K} & D_{R}
      \end{array}\right]
    \begin{bmatrix}
      x \\ g_{K} \\ g_{R}
    \end{bmatrix}
    =
    \begin{bmatrix}
      0 \\ 0 \\ 0
    \end{bmatrix}
    \label{eq:augmented_pencil}
    ,
  \end{equation}
  where we added an additional (redundant) output equation
  \cite[Theorem 3.4]{FP-FD-FB:12a}. 
  % Consider the projection $[z_1,z_2]
  % = [D_{K}D_{K}^\dag y,(I-D_{K}D_{K}^{\dag}) y]$ of the output $y$
  % onto the image and its complement of $D_{K}$, that is, onto a
  % corrupted and a save output.
  %
  % In the following, we will use the save output $z_{1}$ as measurement
  % while making maximal use of the information provided by the
  % corrupted output.
  A multiplication of equation \eqref{eq:augmented_pencil} from the
  left by the projectors $\blkdiag\bigl( I \,,\, D_{K}D_{K}^\dag \,,\,
  (I-D_{K}D_{K}^{\dag}) \bigr)$ yields
  \begin{align*}
    \left[\begin{array}{c|c|c}
        \!sE - A \!&\! - B_{K} \!&\! -B_{R} \!
        \\\hline
        \! D_{K}D_{K}^\dag C \!&\! D_{K} \!&\! D_{K}D_{K}^\dag D_{R} \!
        \\\hline
        \! (I-D_{K}D_{K}^{\dag}) C \!&\! 0 \!&\! (I-D_{K}D_{K}^{\dag}) D_{R} \!
      \end{array}\right]\!
    \begin{bmatrix}
      x \\ g_{K} \\ g_{R}
    \end{bmatrix}
    \!=\!
    \begin{bmatrix}
      0 \\ 0 \\ 0
    \end{bmatrix}.
    % \label{eq:augmented_pencil-2}
  \end{align*}
  The variable $g_{K}$ can be eliminated in the first redundant
  (corrupted) output equation according to
  \begin{align*}
    g_{K} = - D_{K}^{\dag} C x - D_{K}^{\dag} D_{R} g_{R} + (I-D_K^\dag D_K) g_K.
  \end{align*}
%  where $\supscr{g}{hom}_{K} = (I-D_K^\dag D_K) g_K$.
 %
  Thus, $P(s) [ x^{\transpose} \; g_{K}^{\transpose} \;
  g_{R}^{\transpose} ]^{\transpose} = 0$ has no solution, where $P(s)$
  is
  % there exists no attack set $R$ with $|R| \leq |K|$ and $R \neq K$,
  % $s \in \mathbb{C}$, $g_K \in \mathbb{R}^{|K|}$, $g_R \in
  % \mathbb{R}^{|R|}$, and $x \in \mathbb{R}^{n} \setminus \{0\}$ such
  % that $P(s) [ x^{\transpose} \; g_{K}^{\transpose} \;
  % g_{R}^{\transpose} ]^{\transpose} = 0$, where $P(s)$ is the matrix
  % pencil
  \begin{equation*}
    \small
    \left[\begin{array}{c|c|c}
        \!\!sE - A +  B_K D_K^\dag C \!&\! - B_K (I-D_K^\dag D_K) \!&\! - B_{R} + B_{K}D_{K}^{\dag}D_{R} \!\!
        \\\hline
        \!\!(I-D_{K}D_{K}^{\dag}) C \!&\! 0 \!&\! (I-D_{K}D_{K}^{\dag}) D_{R} \!\!
      \end{array}\right]
  \end{equation*}
  The statement follows.
  % Equivalently, the attack sets $K$ and $R$ can be distinguished for
  % system \eqref{eq: cyber_physical_fault_no_D}, and the attack set
  % $K$
  % is identifiable.
\end{proof}

% Without losing generality, in what follows, we assume the descriptor
% system \eqref{eq: cyber_physical_fault} to be given in the form
% \eqref{eq: cyber_physical_fault_no_D}, that is without any feedthrough
% term.
The second design step of our attack identification monitor relies on
the concept of \emph{conditioned invariant subspace}. We refer to \cite{TG:93,GB-GM:91,FLL:90} for a comprehensive
discussion of conditioned invariant subspaces. Let $\Star$ be the
conditioned invariant subspace associated with the system
$(E,A,B,C,D)$, that is, the smallest subspace of the state space
satisfying
\begin{align}
  \label{eq:conditioned}
  \Star =
  \begin{bmatrix}
    A & B
  \end{bmatrix}
  \left(
    \begin{bmatrix}
      E^{-1} \Star \\ \mathbb{R}^m
    \end{bmatrix}
    \cap
    \Ker
    \begin{bmatrix}
      C & D
    \end{bmatrix}
  \right),
\end{align}
and let $L$ be an output injection matrix satisfying
\begin{align}
  \label{eq:conditioned_injection}
  \begin{bmatrix}
    A+LC & B+LD
  \end{bmatrix}
  \begin{bmatrix}
    E^{-1} \Star \\ \mathbb{R}^m
  \end{bmatrix}
  \subseteq \Star.
\end{align}
We transform the descriptor system \eqref{eq:
  cyber_physical_fault_no_D} into a set of canonical coordinates
representing $\Star$ and its orthogonal complement. For a nonsingular
system ($E = I$) such an equivalent state representation can be
achieved by a nonsingular transformation of the form
$Q^{-1}(sI-A)Q$. However, for a singular system different
transformations need to be applied in the domain and codomain such as
$P^{\transpose}(sE-A)Q$ for nonsingular $P$ and $Q$.
% Thus, the domain and codomain spaces cannot be considered as the
% same space. In the time domain this transformation corresponds to
% the coordinate change $x \mapsto Q^{-1} x$ and a left-multiplication
% of the dynamics by $P^{\transpose}$. Such a transformation is
% adopted for the design of our attack identification filter.

\begin{lemma}{\bf\emph{(Input decoupled system
      representation)}}\label{lemma:Input decoupled system
    representation}
\label{lemma:Input decoupled system representation}
  % Consider the system \eqref{eq: cyber_physical_fault} in the form
  % \eqref{eq: cyber_physical_fault_no_D}.
  For the system \eqref{eq: cyber_physical_fault_no_D}, let $\Star$
  and $L$ be as in \eqref{eq:conditioned} and
  \eqref{eq:conditioned_injection}, respectively.
  % Let $\Star$ be the smallest conditioned invariant subspace
  % associated with \eqref{eq: cyber_physical_fault_no_D},
  % $$(E,A-B_K D_K^\dag
  % C,B_K (I-D_K^\dag D_K),(I- D_KD_K^\dag)C)$$
  % and let $L$ be as in \eqref{eq:conditioned_injection}.
  % $(E,A-B_K D_K^\dag C, B_K (I-D_K^\dag D_K), (I-
  % D_KD_K^\dag)C)$-conditioned invariant subspace containing $\Image
  % (B_K (I - D_K^\dag D_K))$, and let $L$ be the associated output
  % injection matrix.
  Define the unitary matrices $ P =
    \begin{bmatrix}
      \Basis(\Star) & \Basis((\Star)^\perp)
    \end{bmatrix}
$ and
$    Q =
    \begin{bmatrix}
      \Basis(E^{-1}\Star) & \Basis((E^{-1}\Star)^\perp)
    \end{bmatrix}
$.
 Then
  \begin{align*}%\label{eq:triangular form}
  \small
    \begin{split}
     & P^\transpose E Q \!=\!
      \begin{bmatrix}
        \tilde E_{11} \!&\! \tilde E_{12}\\
        0 \!&\! \tilde E_{22}
      \end{bmatrix},
      P^\transpose (A-B_K D_K^\dag C+LC) Q \!=\!
      \begin{bmatrix}
        \tilde A_{11} \!&\! \tilde A_{12}\\
        0 \!&\! \tilde A_{22}
      \end{bmatrix},\\
    &  P^\transpose B_K (I-D_K^\dag D_K) \!=\!
      \begin{bmatrix}
        \tilde B_K(t) \\ 0
      \end{bmatrix},
      (I-
  D_KD_K^\dag)C) Q \!=\! 
      \begin{bmatrix}
        \tilde C_1 & \tilde C_2
      \end{bmatrix}.
    \end{split}
  \end{align*}
  The attack set $K$ is identifiable for the descriptor
  system \eqref{eq: cyber_physical_fault} if and only if it is
  identifiable for the descriptor system
  \begin{align}
    \begin{bmatrix}
      \tilde E_{11} & \tilde E_{12}\\
      0 & \tilde E_{22}
    \end{bmatrix}
    \begin{bmatrix}
      \dot x_1(t) \\ \dot x_2(t)
    \end{bmatrix}
    &=
    \begin{bmatrix}
      \tilde A_{11} & \tilde A_{12}\\
      0 & \tilde A_{22}
    \end{bmatrix}
    \begin{bmatrix}
      x_1(t) \\ x_2(t)
    \end{bmatrix}
    +
    \begin{bmatrix}
      \tilde B_K(t) \\ 0
    \end{bmatrix},\nonumber\\
    y(t) &=
    \begin{bmatrix}
      \tilde C_1 & \tilde C_2
    \end{bmatrix}
    \begin{bmatrix}
      x_1(t) \\ x_2(t)
    \end{bmatrix}.
    \label{eq:partitioned system}
  \end{align}
\end{lemma}

\begin{proof}
  Let $\mc L = E^{-1} \Star$ and $\mc M = \Star$. Notice that $(A+LC)
  E^{-1} \Star \subseteq \Star$ by the invariance property of $\Star$
  \cite{FLL:90,TG:93}. It follows that $\mc L$ and $\mc M$ are a pair
  of \emph{right deflating subspaces} for the matrix pair $(A + LC,
  E)$ \cite{KDI:93}, that is, $\mc M = A \mc L + E \mc L$ and
  $\textup{dim}(\mc M) \le \textup{dim}(\mc L)$. The sparsity pattern
  in the descriptor and dynamic matrices $\tilde E$ and $\tilde A$ of
  \eqref{eq:partitioned system} arises by construction of the right
  deflating subspaces $P$ and $Q$ \cite[Eq. (2.17)]{KDI:93}, and the
  sparsity pattern in the input matrix arises due to the invariance
  properties of $\Star$ containing $\Image(B_K)$. The statement
  follows because the output injection $L$, the coordinate change $x
  \mapsto Q^{-1} x$, and the left-multiplication of the dynamics by
  $P^{\transpose}$ does not affect the existence of zero dynamics.
\end{proof}

We call system \eqref{eq:partitioned system} the \emph{conditioned
  system} associated with \eqref{eq: cyber_physical_fault}. For the
ease of notation and without affecting generality, the third and final
design step of our attack identification filter is presented for the
conditioned system \eqref{eq:partitioned system}.

\begin{theorem}{\bf\emph{(Attack identification filter for attack set
      $K$)}}\label{Theorem: attack id filter}
  Consider the \emph{conditioned system} \eqref{eq:partitioned system}
  associated with the descriptor system \eqref{eq:
    cyber_physical_fault}. Assume that the attack set is identifiable,
  the network initial state $x(0)$ is known, and the assumptions (A1)
  through (A3) are satisfied. Consider the {\em attack identification
    filter for the attack signature $(B_K, D_K)$}
    \begin{align}
      \label{eq: identification filter}
      \begin{split}
    % \begin{bmatrix}
    %   E_{11} & E_{12}\\
    %   0 & E_{22}
    % \end{bmatrix}
    % \begin{bmatrix}
    %   \dot w_1(t) \\ \dot w_2(t)
    % \end{bmatrix}
    % &=
    % \begin{bmatrix}
    %   A_{11} & A_{12}\\
    %   0 & \bar A_{22} 
    % \end{bmatrix}
    % \begin{bmatrix}
    %   w_1(t) \\ w_2(t)
    % \end{bmatrix}
    % -
    % \begin{bmatrix}
    %   0 \\ G_2
    % \end{bmatrix}
    % \bar y(t)
%    \\
%    r_K(t) &= (I - C_1 C_1^\dag) C_2 w_2(t) - \bar y(t),
%  \end{split}
        \tilde E_{22} \dot w_2(t) &= (\tilde A_{22} + \tilde G (I -
        \tilde C_1 \tilde C_1^\dag) \tilde C_2) w_2(t)
        - \tilde G \bar y(t),\\
        r_K(t) &= (I - \tilde C_1 \tilde C_1^\dag) \tilde C_2 w_2(t) -
        \bar y(t),\;\;\;\text{ with }\\
        \bar y(t) &= (I - \tilde C_1 \tilde C_1^\dag) \tilde C_2y(t),
      \end{split}
  \end{align}
  where $w_2(0) = x_2(0)$, and $\tilde G$ is such that $(\tilde
  E_{22},\tilde A_{22} + \tilde G (I - \tilde C_1 \tilde C_1^\dag)
  \tilde C_2)$ is Hurwitz. Then $r_{K} (t) = 0$ for all times $t \in
  \real_{\ge 0}$ if and only if $K$ coincides with the attack set.
\end{theorem}
\begin{proof}
  Let $w(t) = [w_1(t)^\transpose
  \; w_2(t)^\transpose]^\transpose$, where $w_1(t)$ obeys 
  \begin{align*}
    \tilde E_{11} \dot w_1 (t) + \tilde E_{12} \dot w_2(t) = \tilde A_{11} w_1(t) + \tilde A_{12}
    w_2(t).
  \end{align*}
Consider the filter error $e(t) = w(t) - x(t)$, and
% \flomargin{$\bar A_{22}$ was not defined before}\fpmargin{You are
% right}
notice that
  \begin{align*}
    \begin{split}
    \begin{bmatrix}
      \tilde E_{11} & \tilde E_{12}\\
      0 & E_{22}
    \end{bmatrix}
    \begin{bmatrix}
      \dot e_1(t) \\ \dot e_2(t)
    \end{bmatrix}
    &=
    \begin{bmatrix}
      \tilde A_{11} & \tilde A_{12}\\
      0 & \bar A_{22}
    \end{bmatrix}
    \begin{bmatrix}
      e_1(t) \\ e_2(t)
    \end{bmatrix}
    -
    \begin{bmatrix}
      \tilde B_K \\ 0
    \end{bmatrix}
    u_K (t),
    \\
    r_K(t) &=
    (I - \tilde C_1 \tilde C_1^\dag) \tilde C_2 e_2(t),
  \end{split}
  \end{align*}
  where $\bar A_{22} = \tilde A_{22} + \tilde G (I - \tilde C_1 \tilde C_1^\dag)
  \tilde C_2)$.
  Notice that $r_K(t)$ is not affected by the input $u_K (t)$, so
  that, since $e_{2}(0) = 0$ due to $w_{2} (0) = x_{2} (0)$, the
  residual $r_K (t)$ is identically zero when $K$ is the attack set.
  In order to prove the theorem we are left to show that for every set
  $R$, with $|R| \le |K|$ and $R \cap K = \emptyset$, every attack
  mode $u_R (t)$ results in a nonzero residual $r_K(t)$. From
  \cite[Theorem 3.4]{FP-FD-FB:12a} and the identifiability
  hypothesis, for any $R \neq K$, there exists no solution to
  % Notice that,
  % by the identifiability hypothesis, the system $(E,A,[B_K
  % B_R],C,D_R)$ has no zero dynamics \cite[Theorem
  % 3.4]{FP-FD-FB:12a}. Equivalently,
  % there exists no attack set $R$ with $|R| \leq |K|$ and $R \neq K$,
  % $s \in \mathbb{C}$, $g_K \in \mathbb{R}^{|K|}$, $g_R \in
  % \mathbb{R}^{|R|}$, and $x \in \mathbb{R}^{n} \setminus \{0\}$ such
  % that
    \begin{align*}
      \left[\begin{array}{cc|c|c}
          \!s \tilde E_{11} - \tilde A_{11}\! & \!s \tilde E_{12} -
          \tilde A_{12}\! & \!  \tilde B_K\! & \! - B_{R1} \! \\
          0 & \!s \tilde E_{22} - \bar A_{22}\! & 0 & \!-
          B_{R2}\!\\\hline \tilde C_1 & \tilde C_2 & 0 & D_R
    \end{array}\right]
    \!\begin{bmatrix}
      x_1 \\ x_2 \\ g_{K} \\ g_{R}
    \end{bmatrix}
    \!=\!
    \begin{bmatrix}
      0 \\ 0 \\ 0 \\ 0
    \end{bmatrix}\!.
%    \label{eq:pencil-2}
  \end{align*}
  A projection of the equation $0 = \tilde C_{1}  x_{1} + \tilde C_{2} 
  x_{2} + D_{R} g_{R}$ onto the image of $\tilde C_{1}$ and its orthogonal
  complement yields
  \begin{align}
   & \small
   \left[\begin{array}{cc|c|c}
      \!s \tilde E_{11} - \tilde A_{11} \!&\! s \tilde E_{12} - \tilde A_{12} \!&\!  B_K \!&\! - B_{R1}\\
      \!0 \!&\! s \tilde E_{22} - \bar  A_{22} \!&\! 0 \!&\! - B_{R2}\!\\\hline
      \!\tilde C_1 \!&\! \tilde C_1 \tilde C_1^\dag \tilde C_2 \!&\! 0 \!&\! \tilde C_1 \tilde C_1^\dag  D_R\!\\
      \!0 \!&\! (I - \tilde C_1 \tilde C_1^\dag) \tilde C_2 \!&\! 0 \!&\! (I - \tilde C_1 \tilde C_1^\dag)  D_R\!\\
    \end{array}\right]\!\!
    \begin{bmatrix}
       x_1 \\  x_2 \\ g_{K} \\ g_{R}
    \end{bmatrix}
    \nonumber\\
    &\small=
    \begin{bmatrix}
      0 & 0 & 0 & 0
    \end{bmatrix}^{\transpose}
    .
    \label{eq:pencil-3}
  \end{align}
  Due to the identifiability hypothesis the set of equations
  \eqref{eq:pencil-3} features no solution $[ x_{1}^{\transpose} \;
  x_{2}^{\transpose} \; g_{K}^{\transpose} \; g_{R}^{\transpose} ]
  ^{\transpose}$ with $[x_{1}^{\transpose} \;
  x_{2}^{\transpose}]^{\transpose} = 0$.
  
  Observe that, for every $ x_2$ and $g_R$, there exists $ x_1 \in
  \Ker(\tilde C_1)^\perp$ such that the third equation of
  \eqref{eq:pencil-3} is satisfied.
  Furthermore, for every $ x_2$ and $g_R$, there exist $ x_1 \in \Ker
  (\tilde C_1)$ and $g_K$ such that the first equation of
  \eqref{eq:pencil-3} is satisfied. Indeed, since $Q E^{-1} \Star =
  [\Image(I) \; 0]^\transpose$ and $P^\transpose \Star = [\Image(I) \;
  0]^\transpose$, the invariance of $\Star$ implies that $\Star = A
  (E^{-1} \Star \cap \Ker (C)) + \Image( B_K)$, or equivalently in new
  coordinates, $\Image(I) = \tilde A_{11} \Ker (\tilde C_1) + \Image
  (\tilde B_K)$. Finally note that $[(s \tilde E_{11} - \tilde A_{11})
  \Ker(\tilde C_1) \; \tilde B_K]$ is of full row rank due to the
  controllability of the subspace $\Star$ \cite{TG:93}. We conclude
  that there exist no vectors $ x_2$ and $g_R$ such that $(s \tilde
  E_{22} - \bar A_{22}) x_2 - B_{R2} g_R = 0$ and $(I - \tilde C_1
  \tilde C_1^\dag) (\tilde C_2 x_2 + D_{R} g_{R}) = 0$ and the
  statement follows.
%   Furthermore, for every $ x_2$ and $g_R$, there exist $ x_1 \in \Ker
%  (\tilde C_1)$ and $g_K$ such that the first equation of
%  \eqref{eq:pencil-3} is satisfied. Indeed, since $Q E^{-1} \Star =
%  [\mc I \; 0]^\transpose$ and $P^\transpose \Star = [\mc I \;
%  0]^\transpose$, the invariance of $\Star$ implies that $\Star = A
%  (E^{-1} \Star \cap \Ker (C)) + \mc B_K$, or equivalently in new
%  coordinates, $\mc I = \tilde A_{11} \Ker (\tilde C_1) + \mc{\tilde
%    B}_K$. Finally note that $[(s \tilde E_{11} - \tilde A_{11})
%  \Ker(\tilde C_1) \; \tilde B_K]$ is of full row rank due to the
%  controllability of the subspace $\Star$ \cite{TG:93}. We conclude
%  that there exist no vectors $ x_2$ and $g_R$ such that $(s \tilde
%  E_{22} - \bar A_{22}) x_2 - B_{R2} g_R = 0$ and $(I - \tilde C_1
%  \tilde C_1^\dag) (\tilde C_2 x_2 + D_{R} g_{R}) = 0$ and the
%  statement follows.
\end{proof}

\IncMargin{.3em}
\begin{algorithm}[t]
   {\footnotesize
    
      \SetKwInOut{Input}{Input}
      \SetKwInOut{Output}{Output}
      \SetKwInOut{Set}{Set}
      \SetKwInOut{Title}{Algorithm}
      \SetKwInOut{Require}{Require}

      \Input{Matrices $E$, $A$, $B_K$, and $D_K$,\;}

      \Require{Identifiability of attack set $K$\;}

    %  \Output{Attack identification filter for $(B_K,D_K)$\;}

      \BlankLine
      \BlankLine
    
      \nl From system \eqref{eq: cyber_physical_fault} define the
      system \eqref{eq: cyber_physical_fault_no_D}\;

      \BlankLine

      \nl Compute $\Star$ and $L$ for system \eqref{eq:
        cyber_physical_fault_no_D} as in \eqref{eq:conditioned} and
      \eqref{eq:conditioned_injection}\;
      
      \BlankLine

      \nl Apply $L$, $P$, and $Q$ as in Lemma \ref{lemma:Input
        decoupled system representation} leading to system
      \eqref{eq:partitioned system}\;

      \BlankLine

      \nl For \eqref{eq:partitioned system}, define $r_K$ and apply
      the output injection $\bar G$ as in \eqref{eq: identification
        filter}.
      
    }
  \caption{\textit{Identification Monitor for $(B_K,D_K)$}}
  \label{alg:identification}
\end{algorithm} 
\DecMargin{.3em}

% In summary, for each attack set $K$ on system \eqref{eq:
%   cyber_physical_fault}, our attack identification filter \eqref{eq:
%   identification filter} is constructed as follows:
% \begin{enumerate}
% \item[(1)] from system \eqref{eq: cyber_physical_fault} define the
%   system \eqref{eq: cyber_physical_fault_no_D},
% \item[(2)] compute $\Star$ and $L$ for system \eqref{eq:
%     cyber_physical_fault_no_D} as in \eqref{eq:conditioned} and
%   \eqref{eq:conditioned_injection},
% \item[(3)] apply the output injection $L$ and the transformation with
%   $P$ and $Q$ to system \eqref{eq: cyber_physical_fault_no_D} leading
%   to system \eqref{eq:partitioned system}, and
% \item[(4)] for the second subsystem of \eqref{eq:partitioned system},
%   define $r_K$ and apply the output injection $\bar G$ as in
%   \eqref{eq: identification filter}.
% \end{enumerate}

% for the implementation of the attack identification filter
% \eqref{eq: identification filter} two output injections through the
% matrices $L$ and $\tilde G$ as well as an output multiplication
% (from the left) by $(I - \tilde C_1 \tilde C_1^\dag) \tilde C_2$
% need to be performed. The output injection $L$ does not appear in
% the filter design \eqref{eq: identification filter} because it has
% already been considered in the transformation \eqref{eq:partitioned
% system}. Second
Our identification procedure is summarized in Algorithm
\ref{alg:identification}. Observe that the proposed attack
identification filter extends classical results concerning the design
of unknown-input fault detection filters. In particular, our filter
generalizes the construction of \cite{MAM-GCV-ASW-CM:89} to descriptor
systems with direct feedthrough matrix. Additionally, we guarantee the
absence of invariant zeros in the residual dynamics. By doing so, our
attack identification filter is sensitive to \emph{every} attack
mode. Notice that classical fault detection filters, for instance
those presented in \cite{MAM-GCV-ASW-CM:89}, are guaranteed to detect
and isolate signals that do not excite exclusively zero
dynamics. Finally, an attack identification filter for the case of
state space or index-one systems is presented in our previous work
\cite{FP-FD-FB:11i}.

\begin{remark}{\bf\emph{(Complexity of centralized identification)}}
  Our centralized identification procedure assumes the knowledge of
  the cardinality $k$ of the attack set, and it achieves
  identification of the attack set by constructing a residual
  generator for $\binom{n+p}{k}$ possible attack sets.
  % among $\binom{n+p}{k}$ possibilities.
  Thus, for each finite value of $k$, our procedure constructs $O
  (n^k)$ filters. If only an upper bound $\bar k$ on the cardinality
  of the attack set is available, identification can be achieved by
  constructing $\binom{n+p}{\bar k}$ filters, and by intersecting the
  attack sets generating zero residuals.  \oprocend
\end{remark}

\begin{remark}{\bf\emph{(Attack identification filter in the presence
      of noise)}}\label{remark:id_noise} Let the dynamics and the
  measurements of the system \eqref{eq: cyber_physical_fault} be
  affected, respectively, by the additive white noise signals $\eta
  (t)$, with $\mathbb{E} [\eta(t) \eta^\transpose(\tau)] = R_{\eta}
  \delta (t - \tau)$, and $\zeta (t)$, with $\mathbb{E} [\zeta(t)
  \zeta^\transpose(\tau)] = R_{\zeta} \delta (t - \tau)$. Let the
  state and output noise be independent of each other. Then, simple
  calculations show that the dynamics and the output of the attack
  identification filter \eqref{eq: identification filter} are
  affected, respectively, by the noise signals
    \begin{align*}
    \hat \eta (t) &= P^\transpose \eta (t) + P^\transpose (L (I - D_K
    D_K^\dag) - B_K D_K^\dag) \zeta (t),\\
    \hat \zeta (t) &= - \bigg(I - \left[ (I - D_K D_K^\dag) C Q_1 \right] \left[ (I - D_K
    D_K^\dag) C Q_1 \right]^\dag  \\&(I - D_K D_K^\dag)  \bigg) \zeta (t),
  \end{align*}
  % \begin{align*}
  %   \hat \eta (t) &= P^\transpose \eta (t) + P^\transpose (L (I - D_K
  %   D_K^\dag) - B_K D_K^\dag) \zeta (t),\\
  %   \hat \zeta (t) &= - \bigg(I - \left[ (I - D_K D_K^\dag) C Q_1 \right] \left[ (I - D_K
  %   D_K^\dag) C Q_1 \right]^\dag  \\&(I - D_K D_K^\dag)  \bigg) \zeta (t),
  % \end{align*}
  where $Q_1 = \Basis (E^{-1} \Star)$. %Let $R_{\hat \eta} = \mathbb{E}
%  [\hat \eta(t) \hat \eta^\transpose(\tau)]$, $R_{\hat \zeta} =
%  \mathbb{E} [\hat \zeta(t) \hat \zeta^\transpose(\tau)]$.
  Define the covariance matrix
  \begin{align*}
    R_{\hat \eta, \hat \zeta} =
    \mathbb{E} \left(
    \begin{bmatrix}
      \hat \eta (t) \\ \hat \zeta (t)
    \end{bmatrix}
    \begin{bmatrix}
      \hat \eta^\transpose (t) & \hat \zeta^\transpose (t)
    \end{bmatrix}
    \right).
  \end{align*}
  Notice that the off-diagonal elements of $R_{\hat \eta, \hat \zeta}$
  are in general nonzero, that is, the state and output noises of the
  attack identification filter are not independent of each other. As
  in the detection case, by using the covariance matrix $R_{\hat \eta,
    \hat \zeta}$, the output injection matrix $\tilde G$ in \eqref{eq:
    identification filter} can be designed to optimize the robustness
  of the residual $r_K (t)$ against noise. A related example is in
  Section \ref{sec:example}.  \oprocend
\end{remark}

We conclude this section by observing that a distributed
implementation of our attack identification scheme is not
practical. Indeed, even if the filters parameters may be obtained via
distributed computation, still $\binom{n+p}{k}$ filters would need to
be implemented to identify an attack of cardinality $k$. Such a
distributed implementation results in an enormous communication effort
and does not reduce the fundamental combinatorial complexity.

% We conclude this section with an observation regarding a distributed
% implementation of the attack identification filter \eqref{eq:
%   identification filter}. The identification filter \eqref{eq:
%   identification filter} can also be implemented in original (sparse)
% coordinates provided that the output injections $L$ and $\tilde G$ as well as
% the output transformation through $\tilde C$ have been
% calculated. Hence, a distributed implementation of the identification
% filter \eqref{eq: identification filter} requires a distributed
% calculation of $\Star$, the corresponding matrices $L$, $\tilde G$, and
% $\tilde C$, as well as a distributed stabilization of the filter
% error analogously to the detection filter \eqref{eq:detection_filter}
% in Section \ref{sec:detection}.
% %
% If all of these objectives can be achieved under assumptions (A6) -
% (A7), then a distributed implementation of the centralized
% identification filter \eqref{eq: identification filter} is generally
% possible. However, still $\binom{n}{k}$ filters need to be implemented
% in a distributed fashion to identify an attack set of cardinality
% $k$. Hence, such a distributed implementation does not resolve the
% fundamental combinatorial complexity and results in an enormous
% communication effort. For these reasons a direct distributed
% implementation of the attack identification filter \eqref{eq:
%   identification filter} will not be further pursued here.

%%% Local Variables: 
%%% mode: latex
%%% TeX-master: "main"
%%% End:
\subsection{Fully decoupled attack identification}
\label{Subsection: Decoupled regional attack identification monitor
  design} In the following sections we develop a distributed attack
identification procedure. Consider the decentralized setup presented
in Section \ref{sec:decentralized} with assumptions (A4)-(A7). The
subsystem assigned to the $i$-th control center is
\begin{align}\label{eq:subsystem_2}
  \begin{split}
    E_i \dot x_i(t) &= A_i x_i(t) + \sum_{j \in \mc N_i^\textup{in}}
    A_{ij} x_j(t) +
    B_{K_i} u_{K_i}(t),\\
    y_i (t) &= C_i x_i (t) + D_{K_i} u_{K_i}(t), \;\; i \in \until{N},
  \end{split}
\end{align}
where $K_i = (K \cap V_i) \cup K_i^\text{p}$ with $K$ being the attack
set and $K_i^\text{p}$ being the set of corrupted measurements in the
region $\subscr{G}{t}^i$.

As a first distributed identification method we consider the fully
decoupled case (no cooperation among control centers). In the spirit of
\cite{MS-YG:92}, the neighboring states $x_j(t)$ affecting $x_{i}(t)$
are treated as unknown inputs ($f_i(t)$) to the $i$-th subsystem:
\begin{align}\label{eq:subsystem_3}
  \begin{split}
    E_i \dot x_i(t) &= A_i x_i(t) + \supscr{B}{b}_i f_i(t) +
    B_{K_i} u_{K_i}(t),\\
    y_i (t) &= C_i x_i (t) + D_{K_i} u_{K_i}(t), \;\; i \in \until{N},
  \end{split}
\end{align}
where $\supscr{B}{b}_i = [A_{i1} \,\cdots\, A_{i,i-1} \, A_{i,i+1}
\,\cdots\, A_{iN}]$. We refer to \eqref{eq:subsystem_3} as to the {\em
  i-th decoupled system}, and we let $\supscr{K}{b}_{i} \subseteq
V_{i}$ be the set of {\em boundary nodes} of \eqref{eq:subsystem_3},
that is, the nodes $j \in V_i$ with $A_{jk} \neq 0$ for some $k \in
\until n \setminus V_{i}$.
% there exists $k \in \until n \setminus V_{i}$ such that $A_{jk} \neq
% 0$.

If the attack identification procedure in Section
\ref{sec:identification_centralized} is designed for the $i$-th
decoupled system \eqref{eq:subsystem_3} subject to unknown inputs
$f_i(t)$ and $u_{K_{i}}(t)$, then a total of only $\sum_{i=1}^{N}
\binom{n_i + p_i}{|K_{i}|} < \binom{n+p}{|K|}$ need to be designed.
% \flomargin{shouldn't it be $\binom{|V_{i}| + p_{i}}{|K_{i}|}$, that
% is, we neglected the output attacks?}\fpmargin{You are right}
Although the combinatorial complexity of the identification problem is
tremendously reduced, this decoupled identification procedure has
several limitations.
%
% By Definitions \ref{undetectable_input} and \ref{unidentifiable_input}
% characterizing detectability and identifiability of attacks, and by
% the results
The following fundamental limitations follow from \cite{FP-FD-FB:12a}:
% are a direct consequence of
% Theorems \cite{FP-FD-FB:12a} for the decoupled system
% \eqref{eq:subsystem_3}:
\begin{enumerate}
	
\item[(L1)] if $(E_{i},A_{i}, B_{K_{i}}, C_{i} , D_{K_{i}})$ has invariant
  zeros, then $K_{i}$ is not detectable by the $i$-th control center;
	
\item[(L2)] if there is an attack set $R_{i}$, with $|R_{i}|\le|K_{i}|$,
  such that $(E_{i},A_{i}, [B_{K_{i}} \; B_{R_{i}}], C_{i} ,
  [D_{K_{i}} \; D_{R_{i}}])$ has invariant zeros, then $K_{i}$ is not
  identifiable by the $i$-th control center;
	
\item[(L3)] if $K_i \not\subseteq \supscr{K}{b}_{i}$ and
  $(E_{i},A_{i}, [\supscr{B}{b}_i \; B_{K_{i}}], C_{i} , D_{K_{i}})$
  has no invariant zeros, then $K_{i}$ is detectable by the $i$-th
  control center; and
\item[(L4)] if $K_i \not\subseteq \supscr{K}{b}_{i}$ and there is no
  attack set $R_{i}$, with $|R_{i}|\le|K_{i}|$, such that
  $(E_{i},A_{i}, [\supscr{B}{b}_i \; B_{K_{i}} \; B_{R_{i}}], C_{i} ,
  [D_{K_{i}} \; D_{R_{i}}])$ has invariant zeros, then $K_{i}$ is
  identifiable by the $i$-th control center.
\end{enumerate}
% \flomargin{I added those guys. Possibly we should removed the
% centralized conditions (i) and (ii)}
Whereas limitations (L1) and (L2) also apply to any centralized attack
detection and identification monitor, limitations (L3) and (L4) arise
by naively treating the neighboring signals as unknown inputs. Since,
in general, the $i$-th control center cannot distinguish between an
unknown input from a safe subsystem, an unknown input from a corrupted
subsystem, and a boundary attack with the same input direction, we
% \flomargin{I added some descriptive adjectives (boundary) and
% (non-internal), what do you think?}\fpmargin{External?}
can further state that
\begin{enumerate}\setcounter{enumi}{4}
\item[(L5)] any (boundary) attack set $K_{i} \subseteq \supscr{K}{b}_{i}$ is not
  detectable and not identifiable by the $i$-th control center, and
\item[(L6)] any (external) attack set $K \setminus K_{i}$
  is not detectable and not identifiable by the $i$-th control center.
\end{enumerate}
We remark that, following our graph-theoretic analysis in \cite
[Section IV]{FP-FD-FB:12a}, the attack $K_{i}$ is generically
identifiable by the $i$-th control center if the number of attacks
$|K_{i}|$ on the $i$-th subsystem is sufficiently small, the internal
connectivity of the $i$-th subsystem (size of linking between unknown
inputs/attacks and outputs) is sufficiently high, and the number of
unknown signals $|\supscr{K}{b}_{i}|$ from neighboring subsystems is
sufficiently small. These criteria can ultimately be used to select an
attack-resilient partitioning of a cyber-physical system.

% In conclusion, the naive fully decoupled case motivated by
% \cite{MS-YG:92} is straightforward and easy to implement but it has
% various limitations, among others it requires safe boundary nodes.

\subsection{Cooperative attack identification}
In this section we improve upon the naive fully decoupled method
presented in Subsection \ref{Subsection: Decoupled regional attack
  identification monitor design} and propose an identification method
based upon a divide and conquer procedure with cooperation.  This
method consists of the following steps.

\noindent \textbf{(S1: estimation and communication)} Each control
center estimates the state of its own region by means of an
\emph{unknown-input observer} for the $i$-th subsystem subject to the
unknown input $\supscr{B}{b}_i f_i(t)$. For this task we build upon
existing unknown-input estimation algorithms (see the Appendix for a
constructive procedure). Assume that the state $x_i (t)$ is
reconstructed modulo some subspace $\mc F_i$.\footnote{For nonsingular
  systems without feedthrough matrix, $\mc F_{i}$ is as small as the
  largest $(A_i,\supscr{B}{b}_i)$-controlled invariant subspace
  contained in $\Ker(C_i)$ \cite{GB-GM:91}.} Let $F_i = \Basis(\mc
F_i)$, and let
% \flomargin{Doesn't $V_{i}$ refer to the set of nodes in partition
% $i$? We have too many $V_{i}$s floating around. Since the subspaces
% are your baby, I leave it up to you change this notation :-)
% }\fpmargin{$W_i$ should be free. But I'll check it again.}
$x_i(t) = \tilde x_i(t) + \hat x_i(t)$, where $\hat x_i(t)$ is the
estimate computed by the $i$-th control center, and $\tilde x_i(t) \in
\mc F_i$. Assume that each control center $i$ transmits the estimate
$\hat x_i(t)$ and the uncertainty subspace $F_i$ to every neighboring
control center.

\noindent \textbf{(S2: residual generation)} Observe that each input
signal $A_{ij} x_j(t)$ can be written as $A_{ij} x_j(t) = A_{ij}
\tilde x_j(t) + A_{ij} \hat x_j(t)$, where $\tilde x_j(t) \in \mc
F_j$. Then, after carrying out step (S1), only the inputs $A_{ij}
\tilde x_j(t) $ are unknown to the $i$-th control center, while the
inputs $A_{ij} \hat x_j(t)$ are known to the $i$-th center due to
communication. Let $\supscr{B}{b}_i F_{i} = [A_{i1}F_1 \, \cdots\,
A_{i,i-1}F_{i-1} \, A_{i,i+1}F_{i+1} \,\cdots\, A_{iN}F_N], $ and
rewrite the signal $\supscr{B}{b}_i \tilde x(t)$ as $\supscr{B}{b}_i
\tilde x(t) = \supscr{B}{b}_i F_i f_i(t)$, for some unknown signal
$f_i(t)$. Then the dynamics of the $i$-th subsystem read as
\begin{align*}
  E_i \dot x_i(t) &= A_i x_i(t) + \supscr{B}{b}_i \hat x(t) +
  \supscr{B}{b}_i F_{i} f_{i}(t) + B_{K_i} u_{K_i} (t).
\end{align*}
%Let
%$
%  \supscr{B}{b}_i V_{i} = [A_{i1}V_1 \, \cdots\, A_{i,i-1}V_{i-1} \,
%  A_{i,i+1}V_{i+1} \,\cdots\, A_{iN}V_N],
%$
%and note that the input $\supscr{B}{b}_i \tilde x(t) = \supscr{B}{b}_i
%V_i v_i(t)$, for some input $v_i(t)$, is unknown while
%$\supscr{B}{b}_i \hat x(t)$ is considered to be known since, in the
%absence of attacks, it is transmitted by the neighboring control
%centers. 
Analogously to the filter presented in Theorem \ref{Theorem: attack id
  filter} for the attack signature $(B_{K},D_{K})$, consider now the
following filter (in appropriate coordinates) for
\eqref{eq:subsystem_3} for the signature $(\supscr{B}{b}_i F_{i},0)$
\begin{align}\label{eq:distr_id_filter}
  \begin{split}
    E_i \dot w_i (t) &= (A_i + L_i C_i) w_i(t) - L y(t) + \supscr{B}{b}_i \bar
    x(t),\\
    r_i(t) &= M w_i(t) - H y(t),
  \end{split}
\end{align}
where $L_i$ is the injection matrix associated with the conditioned
invariant subspace generated by $\supscr{B}{b}_i F_i$, with $(E_i,A_i
+ L_i C_i)$ Hurwitz, and $\bar x(t)$ is the state transmitted to $i$
by its neighbors. Notice that, in the absence of attacks in the
regions $\mc N_i^\text{in}$, we have $\supscr{B}{b}_i \bar x(t) =
\supscr{B}{b}_i \hat x(t)$. Finally, let the matrices $M$ and $H$ in
\eqref{eq:distr_id_filter} be chosen so that the input
$\supscr{B}{b}_i F_{i} f_{i}(t)$ does not affect the residual
$r_i(t)$.\footnote{See Section \ref{sec:identification_centralized}
  for a detailed construction of this type of filter.}
Consider the filter error $e_i(t) = w_i(t) - x_i(t)$, and notice that
\begin{align}\label{eq:error_distr_id}
  E_i \dot e_i (t) &= (A_i + L_i C_i) e_i(t) + \supscr{B}{b}_i (\bar
  x(t) - \hat x(t) ) - B_{K_i} u_{K_i} (t) \nonumber\\ &\;\;\;\;- \supscr{B}{b}_i F_i
  f_i(t), \\
  r_i(t) &= M e_i(t),\nonumber
\end{align}
% Note that, by construction, the input $f_i (t)$ does not affect the
% residual $r_i(t)$.

\noindent \textbf{(S3: cooperative residual analysis)}
We next state a key result for our distributed identification
procedure.

\begin{lemma}{\bf\emph{(Characterization of nonzero
      residuals)}}\label{lemma_regional_identification}
  Let each control center implement the distributed identification
  filter \eqref{eq:distr_id_filter} with $w_i(0) = x_i(0)$. Assume
   that the attack $K$ affects only
  the $i$-th subsystem, that is $K = K_i$. Assume that
  $(E_i,A_i,[\supscr{B}{b}_i F_i \, B_{K_i}],C_i)$ and
  $(E_i,A_i,\supscr{B}{b}_i,C_i)$ have no invariant zeros. Then,
  \begin{enumerate}
    \item $r_i (t) \neq 0$ at some time $t$, and
    \item either $r_j (t) = 0$ for all $j \in \mc N_i^\textup{out}$ at
      all times $t$, or $r_j (t) \neq 0$ for all $j \in \mc
      N_i^\textup{out}$ at some time $t$.
  \end{enumerate}
\end{lemma}
\begin{proof}
  Notice that the estimation computed by a control center is correct
  provided that its area is not under attack. In other words, since $K
  = K_i$, we have that $\supscr{B}{b}_i\hat x(t) = \supscr{B}{b}_i\bar
  x(t)$ in \eqref{eq:error_distr_id}. Since $(E_i,A_i,[\supscr{B}{b}_i
  F_i \, B_{K_i}],C_i)$ has no invariant zeros, statement (i) follows.
  % \flomargin{to make this statement (and other later statements) more
  %   obvious, we should possibly ``include'' $\supscr{B}{b}_i V_i $ in
  %   \eqref{eq:error_distr_id} ?}
  In order to prove statement (ii), consider the following two cases:
  the $i$-th control center provides the correct estimation $\hat x_i
  (t) = \bar x_i(t)$ or an incorrect estimation $\hat x_i(t) \neq \bar
  x_i(t)$. For instance, if $\Image(B_{K_i}) \subseteq
  \Image(\supscr{B}{b}_i)$, that is, the attack set $K_i$ lies on the
  boundary of the $i$-th area, then $\hat x_i (t) = \bar x_i(t)$.
  Notice that, if $\hat x_i (t) = \bar x_i (t)$, then each residual
  $r_j (t)$, $j \neq i$, is identically zero since the associated
  residual dynamics \eqref{eq:error_distr_id} evolve as an autonomous
  system without inputs. Suppose now that $\hat x_i (t) \neq \bar
  x_i(t)$. Notice that $\supscr{B}{b}_i F_i f_i(t) + \supscr{B}{b}_i
  (\hat x(t) - \bar x(t)) \in \Image(\supscr{B}{b}_i)$. Then, since
  $(E_i,A_i,\supscr{B}{b}_i,C_i)$ has no invariant zeros, each
  residual $r_j (t)$ is nonzero for some $t$.
\end{proof}

%As a consequence of Lemma \ref{lemma_regional_identification} the
%region under attack can be correctly identified through a distributed
%procedure. Indeed, by Lemma \ref{lemma_regional_identification}, the
%$i$-th area is safe if \begin{enumerate}
%  \item either the associated residual $r_{i}(t)$ is identically zero, or
%  \item the neighboring areas $j \in \mc N_{i}^\textup{out}$ feature
%    both zero and nonzero residuals $r_{j}(t)$.
%\end{enumerate}
%Consider now the case in which several subsystems are corrupted at the
%same time. Then, if the graphical distance between any two corrupted
%areas is at least $2$, that is, if there are at least two uncorrupted
%areas between any two corrupted areas, corrupted areas can be
%identified via our distributed method. An upper bound on the maximum
%number of identifiable concurrent corrupted areas can
%consequently be derived (see the related \emph{set packing} problem in
%\cite{MRG-DSJ:79}).
As a consequence of Lemma \ref{lemma_regional_identification} the
region under attack can be identified through a distributed
procedure. Indeed, the $i$-th area is safe if either of the following
two criteria is satisfied:
\begin{enumerate}
  \item[(C1)] the associated residual $r_{i}(t)$ is identically zero, or
  \item[(C2)] the neighboring areas $j \in \mc N_{i}^\textup{out}$ feature
    both zero and nonzero residuals $r_{j}(t)$.
\end{enumerate}
Consider now the case of several simultaneously corrupted
subsystems. Then, if the graphical distance between any two corrupted
areas is at least $2$, that is, if there are at least two uncorrupted
areas between any two corrupted areas, corrupted areas can be
identified via our distributed method and criteria (C1) and (C2). An
upper bound on the maximum number of identifiable concurrent corrupted
areas can consequently be derived (see the related \emph{set packing}
problem in \cite{MRG-DSJ:79}).

% \fpmargin{write limits on detection: how many concurrent corrupted
%   regions? Need to define graph describing the areas interconnections}

% {\color{blue} I have to repeat yesterday's comment: in Lemma
%   \ref{lemma_regional_identification}, you proved $\{K = K_{i} \}
%   \implies (i),(ii)$. Now you turn it around an say $(i),(ii) \implies
%   \{K = K_{i} \}$. Why is the equivalence true?  }

\noindent \textbf{(S4: local identification)} Once the corrupted regions
have been identified, the identification method in Section
\ref{sec:identification} is used to identify the local attack set.

\begin{lemma}{\bf\emph{(Local 
      identification)}}\label{lemma:regional}
  Consider the decoupled system \eqref{eq:subsystem_3}. Assume that
  the $i$-th region is under the attack $K_{i}$ whereas the
  neighboring regions $\mc N_i^\textup{out}$ are uncorrupted. Assume
  that each control center $j \in \mc N_i^\textup{in}$ transmits the
  estimate $\hat x_j(t)$ and the uncertainty subspace $F_i$ to the
  $i$-th control center. Then, the attack set $K_{i}$ is identifiable
  by the $i$-th control center if $(E_{i},A_{i}, [\supscr{B}{b}_i
  F_{i} \; B_{K_{i}} \; B_{R_{i}}], C_{i} , [D_{K_{i}} \; D_{R_{i}}])$
  has no invariant zeros for any attack set $R_{i}$, with
  $|R_{i}|\le|K_{i}|$.
\end{lemma} 
\begin{proof}
  Notice that each control center $j$, with $j \neq i$, can correctly
  estimate the state $x_j (t)$ modulo $\mc F_j$. Since this estimation
  is transmitted to the $i$-th control center, the statement follows
  from \cite[Theorem 3.4]{FP-FD-FB:12a}.
\end{proof}

The final identification procedure \textbf{(S4)} is implemented only
on the corrupted regions. Consequently, the combinatorial complexity
of our distributed identification procedure is $\sum_{i=1}^{\ell}
\binom{n_i + p_i}{|K_{i}|}$, where $\ell$ is the number of corrupted
regions. Hence, the distributed identification procedure greatly
reduces the combinatorial complexity of the centralized procedure
presented in Subsection \ref{sec:identification_centralized}, which
requires the implementation of $\binom{n+p}{|K|}$ filters.  Finally,
the assumptions of Lemma \ref{lemma_regional_identification} and Lemma
\ref{lemma:regional} clearly improve upon the limitations (L3) and
(L4) of the naive decoupled approach presented in Subsection
\ref{Subsection: Decoupled regional attack identification monitor
  design}. We conclude this section with an example showing that,
contrary to the limitation (L5) of the naive fully decoupled approach,
boundary attacks $K_{i} \subseteq \supscr{K}{b}_{i}$ can be identified
by our cooperative attack identification method.

\begin{example}{\bf\emph{(An example of cooperative identification)}}
% \begin{figure}[tb]
%   \centering \subfigure[]{
%     \includegraphics[width=.45\columnwidth]{./img/4rombi}
%     \label{fig:4rombi}
%   } \subfigure[]{
%     \;\;\;\includegraphics[width=.45\columnwidth]{./img/RTS-96}
%     \label{fig:rts96}
%   } 
%   \caption[Optional caption for list of figures]{In
%     Fig. \ref{fig:4rombi} shows a network composed of two
%     subsystems. A control center is assigned to each subsystem. Each
%     control center knows only the dynamics of its local subsystem. The
%     state of the blue nodes $\{2,5,7,12,13,15\}$ is continuously
%     measured by the corresponding control center, and the state of the
%     red node $\{3\}$ is corrupted by an attacker. The decoupled
%     identification procedure presented in Subsection \ref{Subsection:
%       Decoupled regional attack identification monitor design} fails
%     at detecting the attack. Instead, by means of our cooperative
%     identification procedure, the attack can be detected and
%     identified via distributed computation.  Fig. \ref{fig:rts96}
%     illustrates the IEEE RTS96 power network
%     \cite{CG-PW-PA-RA-MB-RB-QC-CF-SH-SK-WL-RM-DP-NR-DR-AS-MS-CS:99}. The
%     dynamics of the generators $\{101,102\}$ are affected by an
%     attacker.}
% \end{figure}

  \begin{figure}
    \centering
    \includegraphics[width=.75\columnwidth]{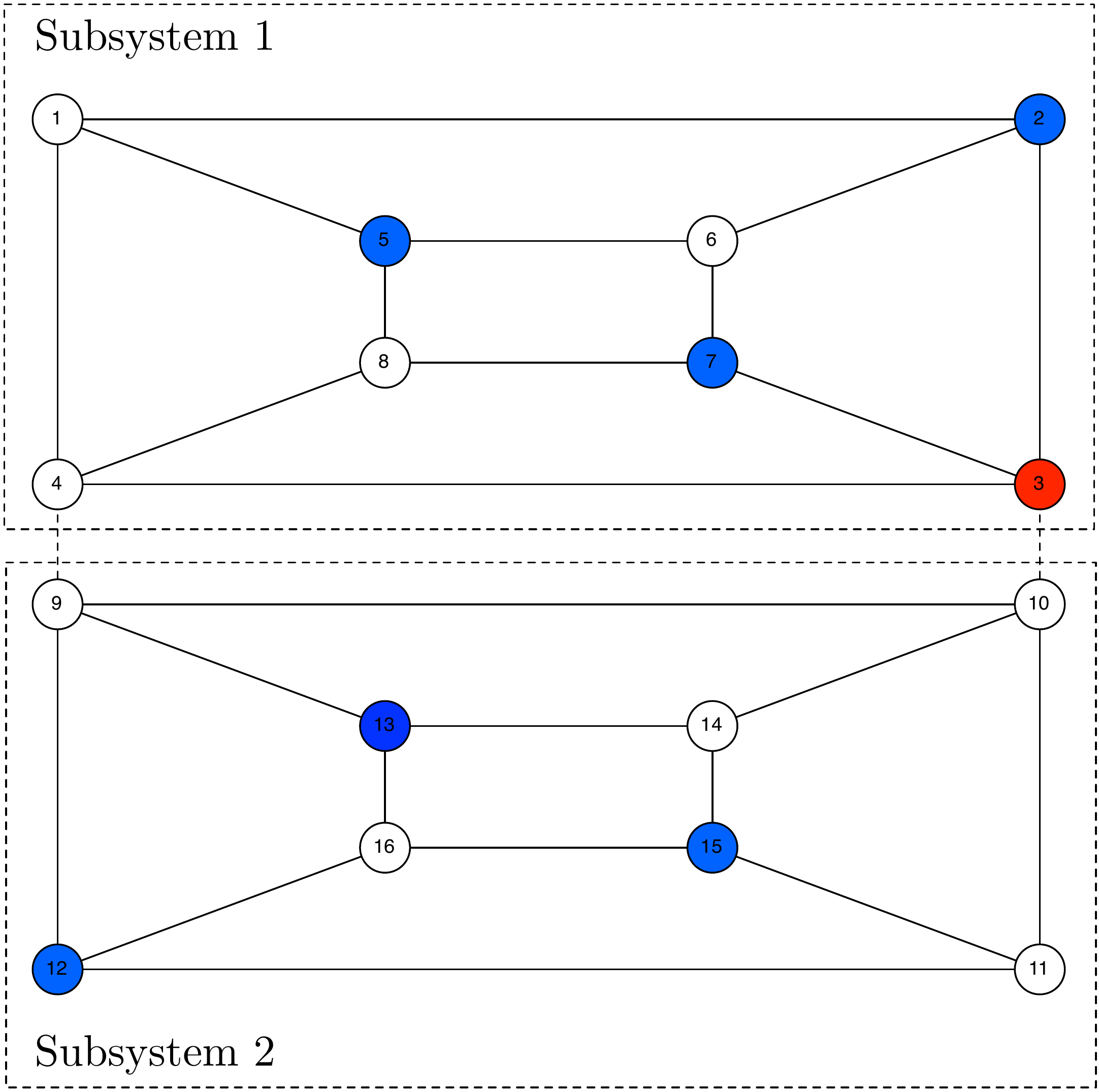}
    \caption{This figure shows a network composed of two subsystems. A
      control center is assigned to each subsystem. Each control center knows only the
      dynamics of its local subsystem. The state of the blue nodes $\{2,5,7,12,13,15\}$ is
      continuously measured by the corresponding control center, and
      the state of the red node $\{3\}$ is corrupted by an attacker. The
      decoupled identification procedure presented in Subsection \ref{Subsection:
        Decoupled regional attack identification monitor design} fails
      at detecting the attack. Instead, by means of our cooperative
      identification procedure, the attack can be detected and
      identified via distributed computation.}
    \label{fig:4rombi}
  \end{figure}
  Consider the sensor network in Fig. \ref{fig:4rombi}, where the
  state of the blue nodes $\{2,5,7,12,13,15\}$ is measured and the state of the red node $\{3\}$ is
  corrupted by an attacker. Assume that the network evolves according
  to nonsingular, linear, time-invariant dynamics. Assume further that
  the network has been partitioned into the two areas $V_1 = \{1,\dots,8\}$
  and $V_2 = \{9,\dots,16\}$ and at most
  one area is under attack. Since $\{3,4\}$ are the boundary
  nodes for the first area, the attack set $K = 3$ is neither
  detectable nor identifiable by the two control centers via the fully
  decoupled procedure in Section \ref{Subsection: Decoupled regional
    attack identification monitor design}.

  Consider now the second subsystem with the boundary nodes
  $\supscr{K}{b}_2 = \{9,10\}$. It can be shown that, generically, the
  second subsystem with unknown input $\supscr{B}{b}_2 f_2(t)$ has no
  invariant zeros; see \cite[Section V]{FP-FD-FB:12a}. Hence, the state of the second
  subsystem can be entirely reconstructed. Analogously, since the
  attack is on the boundary of the first subsystem, the state of the
  first subsystem can be reconstructed, so that the residual $r_2(t)$
  is identically zero; see Lemma \ref{lemma_regional_identification}.

  Suppose that the state of the second subsystem is continuously
  transmitted to the control center of the first subsystem. Then, the
  only unknown input in the first subsystem is due to the attack,
  which is now generically detectable and identifiable, since the
  associated system has no invariant zeros; see Lemma
  \ref{lemma:regional}. We conclude that our cooperative
  identification procedure outperforms the decoupled counterpart in
  Section \ref{Subsection: Decoupled regional attack identification
    monitor design}.  \oprocend
\end{example}

%%% Local Variables: 
%%% mode: latex
%%% TeX-master: "main"
%%% End:

%\clearpage
%\section{Centralized Monitor Design for index-one or nonsingular
%  systems}\label{sec:centralized}
%\input{centralized_monitor_design}

%%\clearpage
%\section{Decentralized Monitor Design based on Structure Preserving
%  Model}\label{sec:decentralized}
%\input{decentralized_monitor_design}

% \newpage
% \section{Monitor Design in the Presence of Noise and Modeling
%   Uncertainties}\label{sec:noise}
% \input{noise}

\section{A case study: the IEEE RTS96 system}\label{sec:example}
\begin{figure}
    \centering
    \includegraphics[width=.75\columnwidth]{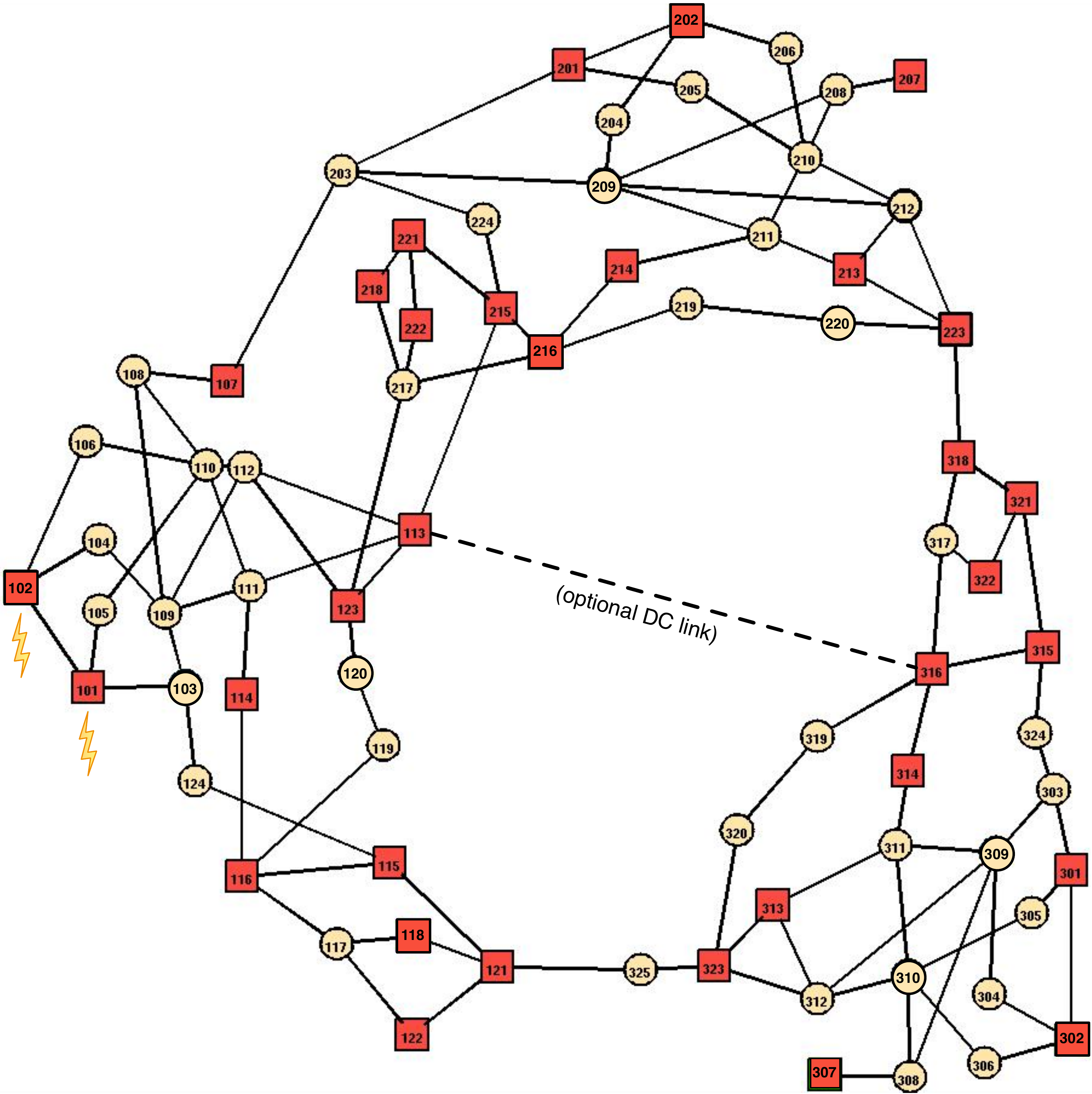}
    \caption{This figure illustrates the IEEE RTS96 power network
      \cite{CG-PW-PA-RA-MB-RB-QC-CF-SH-SK-WL-RM-DP-NR-DR-AS-MS-CS:99}. The
      dynamics of the generators $\{101,102\}$ are affected by an
      attacker.}
    \label{fig:rts96}
\end{figure}

\begin{figure}
   \centering
   \includegraphics[width=.8\columnwidth]{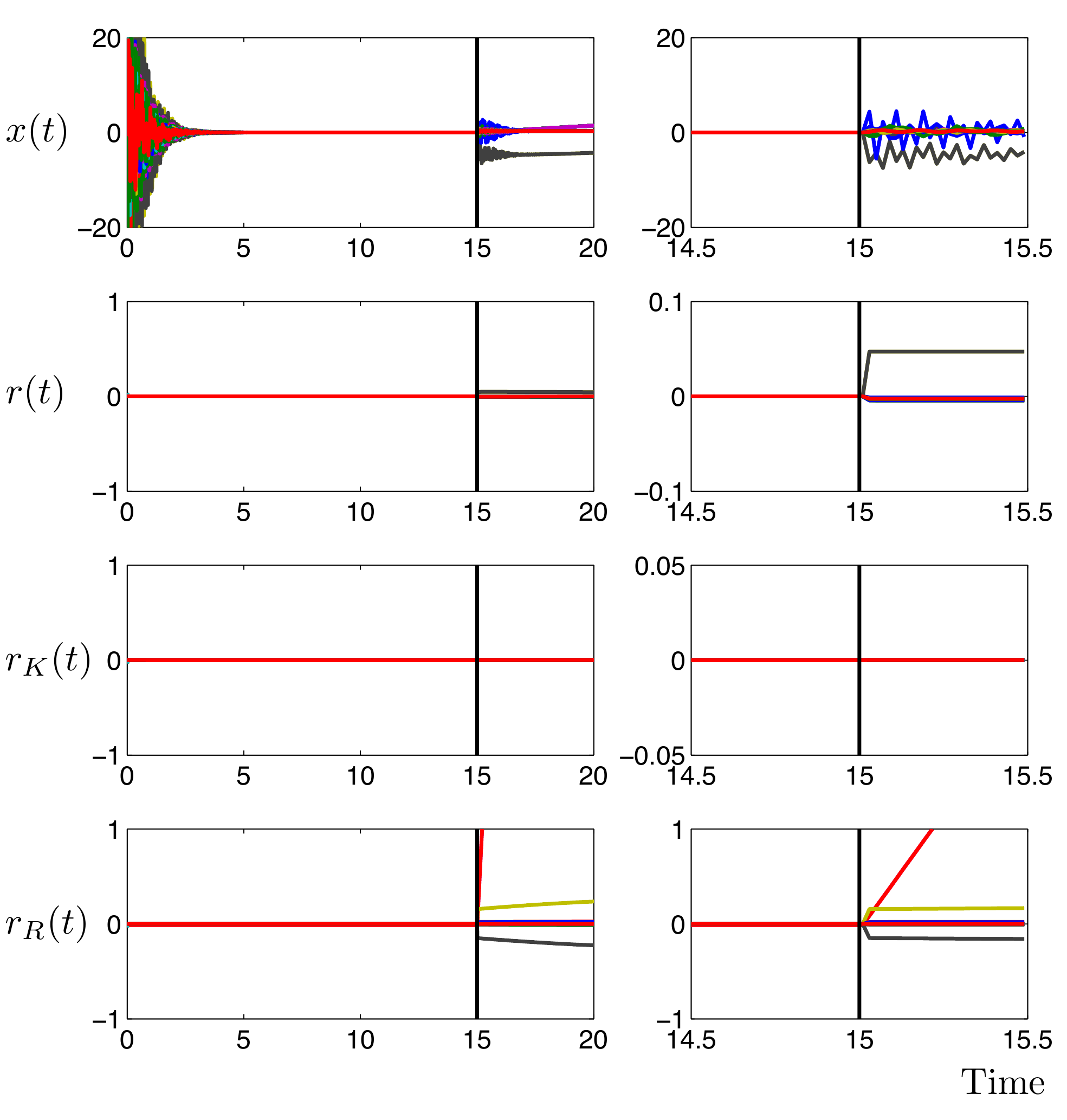}
   \caption{In this figure we report our simulation results for the
     case of linear network dynamics without noise and for the
     proposed detection monitor \eqref{eq:detection_filter} and
     identification monitor \eqref{eq: identification filter},
     respectively. The state trajectory $x(t)$ consists of the
     generators angles and frequencies. The detection residual $r(t)$
     becomes nonzero after time $15$s, and it reveals the presence of
     the attack. The identification residual $r_K(t)$ is identically
     zero even after time $15$s, and it reveals that the attack set is
     $K = \{101,102\}$. The identification residual $r_R (t)$ is
     nonzero after time $15$s, and it reveals that $R$ is not the
     attack set.}
   \label{fig:nominal}
\end{figure}

In this section we apply our centralized attack detection and
identification methods to the IEEE RTS96 power network
\cite{CG-PW-PA-RA-MB-RB-QC-CF-SH-SK-WL-RM-DP-NR-DR-AS-MS-CS:99}
illustrated in Fig. \ref{fig:rts96}. In particular, we first consider
the nominal case, in which the power network dynamics evolve as
nominal linear time-invariant descriptor system, as described in
\cite[Section II.C]{FP-FD-FB:12a}. Second, we consider the case of
additive state and measurement noise, and we show the robustness of
the attack detection and identification monitors. Third, we consider
the case of nonlinear differential-algebraic power network dynamics
and show the effectiveness of our methods in the presence of unmodeled
nonlinear dynamics.

For our numerical studies, we assume the angles and frequencies of
every generator to be measured. Additionally, we let the attacker
affect the angles of the generators $\{101,102\}$ with a random signal
starting from time $15\text{s}$. Since the considered power network
dynamics are of index one, the filters are implemented using the
nonsingular Kron-reduced system representation \cite[Section
III.D]{FP-FD-FB:12a}. The results of our simulations are in
Fig. \ref{fig:nominal}, Fig. \ref{fig:noise}, and
Fig. \ref{fig:nonlinear}. In conclusion, our centralized detection and
identification filters appears robust to state and measurements noise
and unmodeled dynamics.

\begin{figure}
   \centering
   \includegraphics[width=.8\columnwidth]{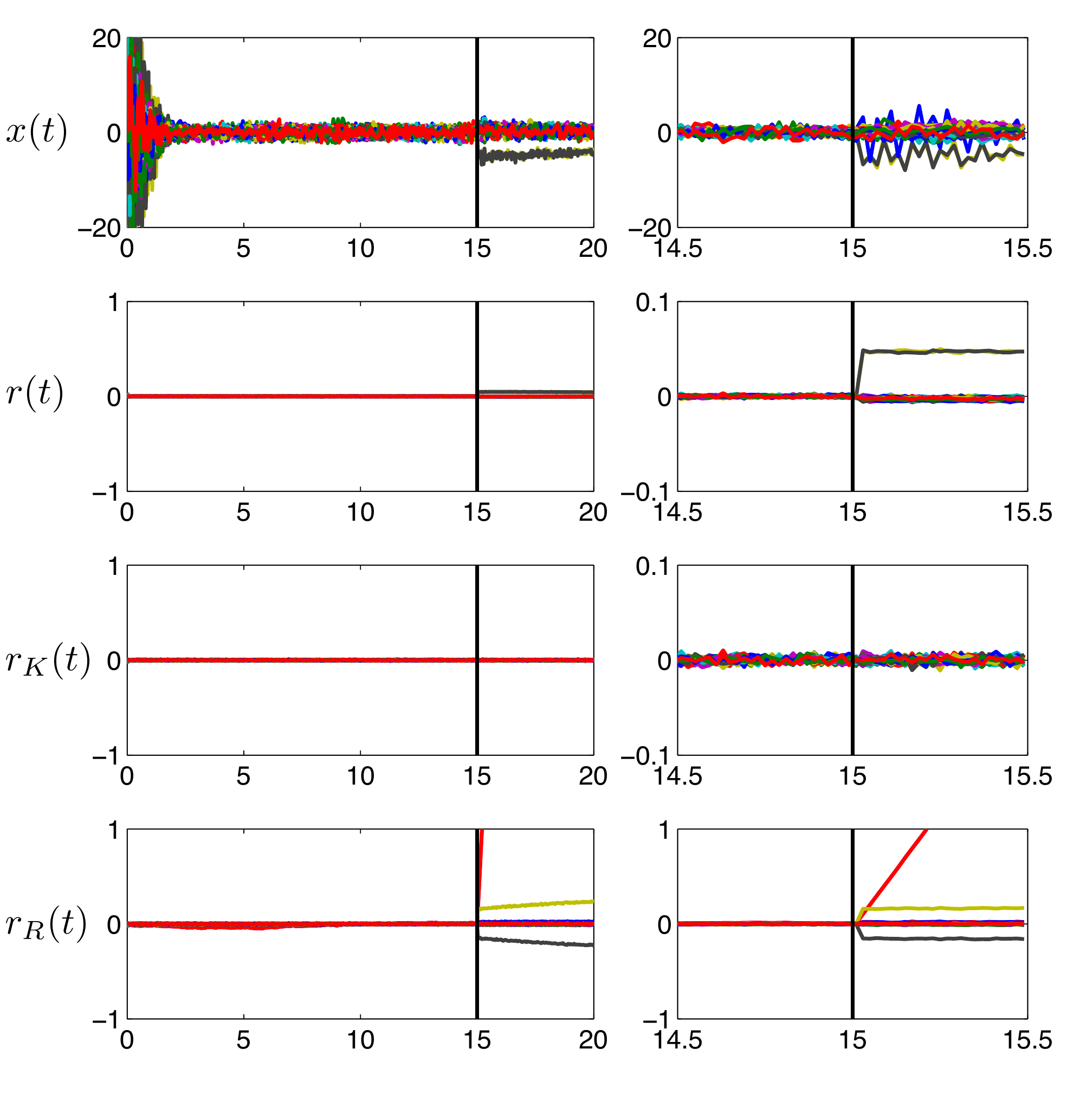}
   \caption{In this figure we report our simulation results for the
     case of linear network dynamics driven by state and measurements
     noise. For this case, we choose the output injection matrices of
     the detection and identification filters as the corresponding
     optimal Kalman gain (see Remark \ref{remark:det_noise} and Remark
     \ref{remark:id_noise}). Due to the presence of noise, the
     residuals deviate from their nominal behavior reported in
     Fig. \ref{fig:nominal}. Although the attack is clearly still
     detectable and identifiable, additional statistical tools such as
     hypothesis testing \cite{MB-IVN:93} may be adopted to analyze the
     residuals $r(t)$, $r_K(t)$, and $r_R (t)$.}
   \label{fig:noise}
\end{figure}

\begin{figure}
   \centering
   \includegraphics[width=.8\columnwidth]{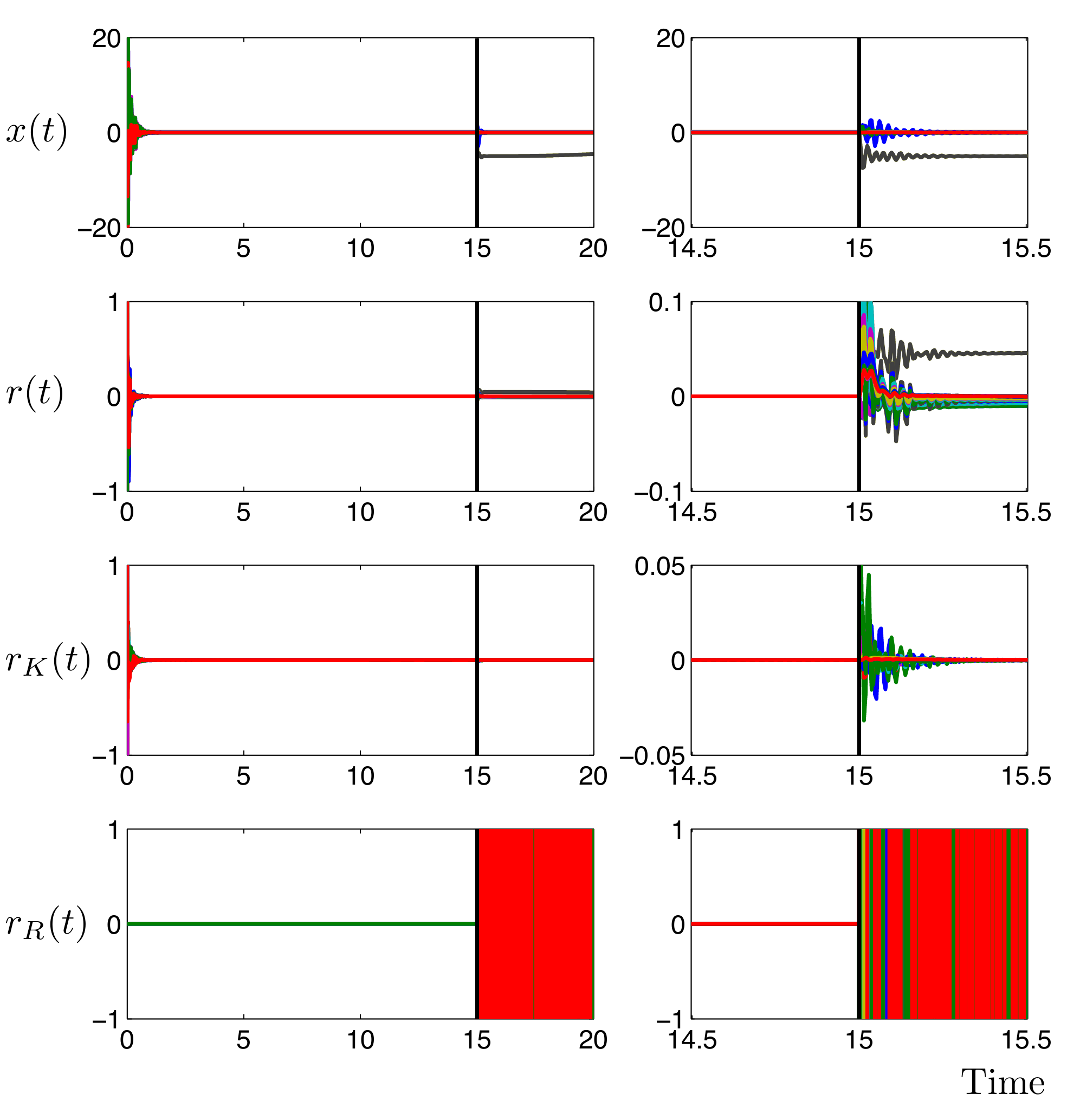}
   \caption{In this figure we report our simulation results for the
     case of nonlinear network dynamics without noise. For this case,
     the detection and identification filters are designed for the
     {\em nominal linearized dynamics} with output injection matrices
     as the corresponding optimal Kalman gain (see Remark
     \ref{remark:det_noise} and Remark \ref{remark:id_noise}). Despite
     the presence of unmodeled nonlinear dynamics, the residuals
     reflect their nominal behavior reported in
     Fig. \ref{fig:nominal}.}
   \label{fig:nonlinear}
\end{figure}

% In Fig. \ref{fig:nominal} we report the residuals for the nominal
% case. The detection and the identification filters are designed,
% respectively, as in Theorem \ref{eq:detection_filter} and Theorem
% \ref{eq: identification filter}. Observe that the detection residual
% reveals the presence of an attack after time $15$. Moreover, since the
% identification residual is exactly zero, the attack set is correctly
% identified.

% In Fig. \ref{fig:noise}, we report the residuals for the noisy nominal
% case. The detection and the identification filters have the same
% structure as in the nominal case, with the exception that the output
% injection matrices are chosen as the optimal Kalman gain (see Remark
% \ref{remark:det_noise} and Remark \ref{remark:id_noise}). Observe
% that, although the residuals deviate from the nominal case, the attack
% is still detectable and identifiable.

% In Fig. ?? we report the residuals for case of a nonlinear model plant
% mismatch, that is, the attack detection and identification filters are
% designed for the {\em nominal linearized dynamics} whereas the {\em
%   original physical dynamics} are nonlinear and their measurements
% drive the filters. As it can be seen in Fig. ??, both the attack and
% identification filters designed in the previous paragraph are robust
% to this nonlinear model plant mismatch. Besides, initial transients in
% the attack identification filter (conditioned upon the correct
% hypothesis), our filters detect and identify the correct attack.

%%% Local Variables:
%%% mode: latex
%%% TeX-master: "main"
%%% End:

%\clearpage
\section{Conclusion}\label{sec:conclusion}
For cyber-physical systems modeled by linear time-invariant descriptor
systems, we proposed attack detection and identification monitors. In
particular, for the detection problem we developed both centralized and
distributed monitors. These monitors are optimal, in the sense that
they detect every detectable attack. For the attack identification
problem, we developed an optimal centralized monitor and a sub-optimal
distributed method. Our centralized attack identification monitor
relies upon a combinatorial machinery. Our distributed attack
identification monitor, instead, is computationally efficient and
achieves guaranteed identification of a class of attacks, which we
characterize. Finally, we provided several examples to show the
effectiveness and the robustness of our methods against uncertainties
and unmodeled dynamics.

%%% Local Variables: 
%%% mode: latex
%%% TeX-master: "main"
%%% End:

\renewcommand{\theequation}{A-\arabic{equation}}
  % redefine the command that creates the equation no.
\setcounter{equation}{0}  % reset counter 
\section*{APPENDIX}  % use *-form to suppress numbering
%\section{Appendix}\label{appendix}
In this section we present an algebraic technique to reconstruct the
state of a descriptor system. Our method builds upon the results
presented in \cite{FJB-TF-WP-GZ:11}. Consider the descriptor model
\eqref{eq: cyber_physical_fault} written in the form (see
\cite[Section IV.C]{FP-FD-FB:12a})
\begin{align}
\label{eq:partitioned_descriptor_system}
\begin{split}
  \dot x_1 (t)&= A_{11} x_1(t) + A_{12} x_2(t) + B_1 u(t)\,,\\
  0 &= A_{21} x_{1}(t) + A_{22} x_{2}(t) + B_{2} u(t)\,, \\
  y(t) &= C_{1}x_{1}(t) + C_{2} x_{2}(t)  + Du(t)
  \,.
\end{split}
\end{align}
We aim at characterizing the largest subspace of the state
space of \eqref{eq:partitioned_descriptor_system} that can be
reconstructed through the measurements $y(t)$. 
% Recall from
% \cite{TG:93,FLL:90} that the subspace of the state space of
% \eqref{eq:associated_nonsingular_system} that cannot be reconstructed
% coincides with the largest subspace $\Vtar = \Image([V_1^\transpose \;
% V_2^\transpose]^\transpose)$ satisfying
% \begin{align}\label{eq:lewis_controlled}
%   \begin{bmatrix}
%     A_{11} & A_{12}\\
%     A_{21} & A_{22}\\
%     C_1 & C_2
%   \end{bmatrix}
%   \begin{bmatrix}
%     V_1 \\ V_2
%   \end{bmatrix}
%   =
%   \begin{bmatrix}
%     V_1\\ 
%     0\\
%     0
%   \end{bmatrix}
%   F
%   -
%   \begin{bmatrix}
%     B_1\\ 
%     B_2\\
%     D
%   \end{bmatrix}
%   G,
% \end{align}
% for some matrices $F$ and $G$.
Consider the associated nonsingular system
\begin{align}
    \dot{\tilde x}_1 (t)&= A_{11} \tilde x_1(t) + B_1 \tilde u(t) + A_{12}
    \tilde x_2(t),
    \label{eq:associated_nonsingular_system} \\
    \tilde y(t) & =
    \begin{bmatrix}
      \tilde y_1 (t) \\ \tilde y_2 (t)
    \end{bmatrix}
    =
    \begin{bmatrix}
      A_{21}\\
      C_1
    \end{bmatrix}
    \tilde x_1(t) + 
    \begin{bmatrix}
      A_{22}  & B_2\\
      C_2 & D
    \end{bmatrix}
    \begin{bmatrix}
      \tilde x_2(t)\\
      \tilde u(t)
    \end{bmatrix}
    \,.
	\nonumber
\end{align}
Recall from \cite[Section 4]{GB-GM:91} that the state of the system
\eqref{eq:associated_nonsingular_system} can be reconstructed modulo
its largest controlled invariant subspace $\Vtar_1$ contained in the
null space of the output matrix. %Let $\Vtar_1 = \Image(\tilde V_1)$.

% satisfying
% \begin{align}\label{eq:lewis_controlled2}
%   \begin{bmatrix}
%     A_{11} \\ A_{21} \\ C_1
%   \end{bmatrix}
%   \tilde V_1
%   =
%   \begin{bmatrix}
%     \tilde V_1 \\ 0 \\ 0
%   \end{bmatrix}
%   \tilde F -
%   \begin{bmatrix}
%     A_{12} & B_1 \\ A_{22} & B_2 \\ C_2 & D
%   \end{bmatrix}
%   \begin{bmatrix}
%     \tilde G_1 \\ \tilde G_2
%   \end{bmatrix}
%   .
% \end{align}

\begin{lemma}{\bf\emph{(Reconstruction of the state
      $x_{1}(t)$)}}\label{lemma:reconstruction_x1}
  Let $\Vtar_1$ be the largest controlled invariant subspace of the
  system \eqref{eq:associated_nonsingular_system}. The state $x_1(t)$
  of the system \eqref{eq:partitioned_descriptor_system} can be
  reconstructed only modulo $\Vtar_1$ through the measurements $y(t)$.
\end{lemma}
\begin{proof}
  We start by showing that for every $x_1 (0) \in \Vtar_1$ there exist
  $x_2(t)$ and $u(t)$ such that $y(t)$ is identically zero. Due to the
  linearity of \eqref{eq:partitioned_descriptor_system}, we conclude
  that the projection of $x_1(t )$ onto $\Vtar_1$ cannot be
  reconstructed. Notice that for every $\tilde x_1 (0)$, $\tilde
  x_2(t)$, and $\tilde u(t)$ yielding $\tilde y_1(t) = 0$ at all
  times, the state trajectory $[\tilde x_1 (t) \; \tilde x_2(t)]$ is a
  solution to \eqref{eq:partitioned_descriptor_system} with input
  $u(t) = \tilde u(t)$ and output $y(t) = \tilde y_2(t)$. Since for
  every $\tilde x_1(0) \in \Vtar_1$, there exists $\tilde x_2(t)$ and
  $\tilde u(t)$ such that $\tilde y(t)$ is identically zero, we
  conclude that every state $x_1 (0) \in \Vtar_1$ cannot be
  reconstructed.
  
  We now show that the state $x_1 (t)$ can be reconstructed modulo
  $\Vtar_1$. Let $x_1 (0)$ be orthogonal to $\Vtar_1$, and let
  $x_1(t)$, $x_2(t)$, and $y(t)$ be the solution to
  \eqref{eq:partitioned_descriptor_system} subject to the input
  $u(t)$. Notice that $\tilde x_1 (t) = x_1(t)$, $\tilde y_1 (t) = 0$,
  and $\tilde y_2(t) = y(t)$ is the solution to
  \eqref{eq:associated_nonsingular_system} with inputs $\tilde x_2 (t)
  = x_2(t)$ and $\tilde u (t) = u(t)$. Since $\tilde x_1(0)$ is
  orthogonal to $\Vtar_1$, we conclude that $\tilde x_1(0) = x_1 (0)$,
  and in fact the subspace $(\Vtar)^\perp$, can be reconstructed
  through the measurements $\tilde y_2(t) = y(t)$.
\end{proof}

In Lemma \ref{lemma:reconstruction_x1} we show that the state $x_1
(t)$ of \eqref{eq:partitioned_descriptor_system} can be reconstructed
modulo $\Vtar_1$. We now show that the state $x_2 (t)$ can generally not be completely reconstructed.

\begin{lemma}\emph{\bf (Reconstruction of the state $x_{2}(t)$)}\label{lemma:reconstruction_x2}
  Let $\Vtar_1 = \Image(V_1)$ be the largest controlled invariant
  subspace of the system \eqref{eq:associated_nonsingular_system}. The
  state $x_2(t)$ of the system
  \eqref{eq:partitioned_descriptor_system} can be reconstructed only
  modulo $\Vtar_2 = A_{22}^{-1} \Image ( [A_{21}V_1 \; B_2 ] )$.
\end{lemma}
\begin{proof}
  Let $x_1 (t) = \bar x_1 (t) + \hat x_1 (t)$, where $\bar x_1 (t)
  \in \Vtar_1$ and $\hat x_1 (t)$ is orthogonal to $\Vtar_1$. From
  Lemma \ref{lemma:reconstruction_x1}, the signal $\hat x_1(t)$ can be
  entirely reconstructed via $y(t)$. Notice that 
  \begin{align*}
    0 &= A_{21} x_1(t) + A_{22} x_2(t) + B_2 u(t),\\
    &= A_{21} V_1 v_1(t) + A_{21} \hat x_1(t) + A_{22} x_2(t) + B_2 u(t).
  \end{align*}
  Let $W$ be such that $\Ker (W) = \Image([A_{21}V_1 \; B_2 ])$. Then,
  $0 = W A_{21} \hat x_1(t) + W A_{22} x_2(t)$, and hence $x_2 (t) =
  \bar x_2(t) + \hat x_2(t)$, where $\hat x_2 (t) = (W A_{22})^\dag W
  A_{21} \hat x_1(t)$, and $\bar x_2 (t) \in \Ker (W A_{22}) =
  A_{22}^{-1} \Image ( [A_{21}V_1 \; B_2 ] )$. The statement follows.
\end{proof}

To conclude the paper, we remark the following points. First, our
characterization of $\Vtar_1$ and $\Vtar_2$ is equivalent to the
definition of \emph{weakly unobservable} subspace in \cite{TG:93}, and
of maximal \emph{output-nulling} subspace in \cite{FLL:90}. Hence, we
 proposed an optimal state estimator for our distributed attack
identification procedure, and the matrix $V_i$ in {\bf (S1: estimation
  and communication)} can be computed as in
\cite{TG:93,FLL:90}. Second, a reconstruction of $x_1 (t)$ modulo
$\Vtar_1$ and $x_2(t)$ modulo $\Vtar_2$ can be obtained through
standard algebraic techniques \cite{GB-GM:91}. Third and finally,
Lemma \ref{lemma:reconstruction_x1} and Lemma
\ref{lemma:reconstruction_x2} extend the results in
\cite{FJB-TF-WP-GZ:11} by characterizing the subspaces of the state
space that can be reconstructed with an algebraic method by processing
the measurements $y(t)$ and their derivatives.

%%% Local Variables: 
%%% mode: latex
%%% TeX-master: "main"
%%% End:

\bibliographystyle{IEEEtran}
\bibliography{alias,Main,FB}

% \bibliographystyle{unsrt}
% \begin{thebibliography}{1}

% \bibitem{CSMguide}
% D.S. Bernstein. ``CSM Author Guide.'' Available online:
% \url{http://www.ieeecss.org/PAB/csm/files/CSMAuthorGuide.pdf}

% \bibitem{IEEEtran}
% M. Shell.  ``IEEE Transactions class file.''  Available online:
% \url{http://www.ctan.org/tex-archive/macros/latex/contrib/IEEEtran/}

% \bibitem{ctan}
% ``The Comprehensive \TeX\ Archive Network.''
% Available online: \url{http://www.ctan.org/}

% \bibitem{csm.sty}
% K.H. Lundberg.  ``LaTeX Style File for IEEE Control Systems Magazine.''
% Available online: \url{http://www.mit.edu/klund/www/csm/}

% \end{thebibliography}

% \section{Author's Bio}

% \noindent {\bf Fabio Pasqualetti} ...

\end{document}